\documentclass{amsart}
\usepackage{fullpage}
\usepackage[foot]{amsaddr}
\usepackage{tikz}
\usetikzlibrary{decorations.pathreplacing,calligraphy}
\usepackage{hyperref}
\usepackage{amsmath}
\usepackage{pgfmath}
\usepackage{amsthm}
\usepackage{amsfonts}
\usepackage{amssymb}
\usepackage{graphicx}
\usepackage{float}
\usepackage{ytableau}
\usepackage[all]{xy}
\usepackage[normalem]{ulem}

\newcommand{\stkout}[1]{\ifmmode\text{\sout{\ensuremath{#1}}}\else\sout{#1}\fi}

\newtheorem{theorem}{Theorem}[subsection]
\newtheorem*{theorem*}{Theorem}
\newcounter{tmp}
\newtheorem{proposition}{Proposition}

\newtheorem{corollary}[theorem]{Corollary}
\newtheorem{definition}[proposition]{Definition}

\newtheorem{lemma}[proposition]{Lemma}

\newtheorem{remark}{Remark}

\newcommand{\Ind}{\mathrm{Ind}}

\newcommand{\ZZ}{\mathbb{Z}}

\newenvironment{spmatrix}{\left(\begin{smallmatrix}}{\end{smallmatrix}\right)}

\title{Some Combinatorics in the Cancellation of Poles of Eisenstein Series for $GL(n,\mathbb{A}_\mathbb{Q})$}
\author{Zhuohui Zhang\\Tel Aviv University}
\email{zhuohui.zhang.ru@gmail.com}

\begin{document}

\begin{abstract}
    The cancellations of poles of degenerate Eisenstein series were studied by Hanzer and Mui\'{c}. This paper generalized the result to Eisenstein series constructed from inducing two Speh representations $\Delta(\tau,m)|\cdot|^{s_1}\otimes\Delta(\tau,n)|\cdot|^{s_2}$ for the group $GL(m+n,\mathbb{A}_\mathbb{Q})$ for self-dual cuspidal automorphic representation $\tau$ by describing the combinatorics of the relevant Weyl group coset.
\end{abstract}
\maketitle
\section{Introduction}
In \cite{moeglin1989spectre}, M\oe glin and Waldspurger described the residual automorphic spectrum of $GL(N)$. A residual automorphic representation can always be realized as a \emph{generalized Speh representation} $\Delta(\tau, n)$, where $\tau$ is an irreducible unitary cuspidal automorphic representation of the group $GL(a)$ with $a$ satisfying $N = an$. For convenience, we will assume the cuspidal automorphic representation $\tau = \bigotimes_v \tau_v$ has a unitary central character, and $\tau_v$ is tempered at each local place $v$. The automorphic representation $\Delta(\tau,n)$ is a global Langlands quotient of a principal series representation, and can be realized as the automorphic representation generated by the residue of an Eisenstein series. %In \cite{ginzburg2006certain} and \cite{jiang2013fourier}, it is proven that the Whittaker support of $\Delta(\tau,n)$ consists of a single nilpotent orbit corresponding to the partition $[a^n]$. 
For any partition $\underline{N} = (N_1,\ldots,N_r)$ of $N$, denoting the standard parabolic subgroup with Levi subgroup $M_{\underline{N}}$ isomorphic to $GL(N_1)\times\ldots GL(N_r)$ by $P_{\underline{N}} = M_{\underline{N}}U_{\underline{N}}$, we can choose cuspidal automorphic representations $\tau_1,\ldots,\tau_r$ of $GL(a_i)$, where $a_in_i = N_i$, and construct the principal series representation
\[
    I(\underline{\tau},\underline{s}) = \Ind_{P_{\underline{N}}}^G\left(\Delta(\tau_1,n_1)|\cdot|^{s_1}\otimes \ldots\otimes \Delta(\tau_r, n_r)|\cdot|^{s_r}\right).
\]
%For $\tau_i$ mutually distinct and tempered at all local places, assuming $L(\frac{1}{2},\tau_i\times\hat{\tau}_j)\neq 0$ and $\underline{s}$ lying in the domain in which the Eisenstein series $E(\cdot,\underline{s}):  I(\underline{\tau},\underline{s})\rightarrow \mathcal{A}_{G_N}$ has no poles, \cite{liu2020top} showed that the Whittaker support of $I(\underline{\tau},\underline{s})$ corresponds to the partition $[a_1^{n_1}] + \ldots + [a_r^{n_r}]$.\\
\indent This paper investigates the poles of the Eisenstein series constructed from a section in $I(\underline{\tau},\underline{s})$ with isomorphic cuspidal data $\tau_1 = \tau_2 = \tau$ in the case $r = 2$.

\begingroup
    \setcounter{tmp}{\value{theorem}}
    \setcounter{theorem}{0}
    \renewcommand\thetheorem{\Alph{theorem}}
\begin{theorem}\label{thma}
    Fixing an irreducible unitary self-dual cuspidal automorphic representation $\tau$, we construct the principal series
    \[
        I(\tau,\underline{s}) = \Ind_{P_{\underline{N}}}^G\left(\Delta(\tau,m)|\cdot|^{s_1}\otimes \Delta(\tau, n)|\cdot|^{s_2}\right)
    \]
    where $\underline{N} = (m,n)$. In the region $s_1-s_2\geq 0$, the Eisenstein series 
    \[
        E(\cdot,\underline{s}): I(\tau,\underline{s})\longrightarrow\mathcal{A}_{G_N}
    \]
    has poles of maximal order 1 at $s_1-s_2 \in \frac{m+n}{2} - \{0,1,\ldots,\min\{m,n\}-1\}$.
\end{theorem}
The proof of this theorem is a generalization of the combinatorial method in \cite{hanzer2015images} to cuspidal automorphic induction data. The residues of this Eisenstein series are described in the following corollary:
\begingroup
    \setcounter{tmp}{\value{theorem}}
    \setcounter{theorem}{1}
    \renewcommand\thetheorem{\Alph{theorem}}
\begin{corollary}\label{corb}
    The residue of Eisenstein series $\mathrm{Res}_{s = \frac{m+n}{2}-\sigma} E(\cdot,\underline{s})$ defines an intertwining operator sending $I(\tau,\underline{s})$ to the irreducible representation $\mathrm{Ind}_{P_{[m,n]}}^{G_{m+n}}\left(\Delta(\tau,m+n-\sigma)|\cdot|^{-\frac{m-n}{2}}\boxtimes \Delta(\tau,\sigma)|\cdot|^{\frac{m-n}{2}}\right)$.
\end{corollary}

\section{Principal Series and Eisenstein Series}
Throughout this section, we denote the space of automorphic functions on $G_n = GL(na)$ by $\mathcal{A}_{G_n}$. The standard parabolic subgroup with Levi subgroup isomorphic to $GL(a)^{\times n}$ is denoted by $P_{n} = M_{n}U_{n}$.

\subsection{Induction from Cusp Forms}\label{cuspinduction}
The cuspidal automorphic representation $\tau$ on the group $GL(a)$ can be decomposed as a restricted tensor product $\tau = \bigotimes_{v}\tau_v$ of irreducible unitary representations $\tau_v$ over local fields $\mathbb{Q}_v$. One can choose a pure tensor vector $f = \otimes_v f_v$ which realizes the completed automorphic $L$-function as the normalization factor of the intertwining operator (for the construction of such a vector c.f. \cite{shahidi1984fourier}). Thus, if we consider the tensor product $\bigotimes_{k=1}^{n}\tau|\cdot|^{\nu_i}$ as a representation on the group $GL(a)^{\times n}$, we can still choose a pure tensor $\bigotimes_{k=1}^n\left(\otimes_v f_{k,v}\right)\in\tau$ satisfying the same property. This vector can be realized as an automorphic function $f_{\underline{\nu}}\in \bigotimes_{k=1}^n \mathcal{A}_{G_a}$. By the Iwasawa decomposition of $G_n = GL_{na}$, we can choose a section $\tilde{f}_{\underline{\nu}}\in \Ind_{P_n}^G\left(\bigotimes_{k=1}^{n}\tau|\cdot|^{\nu_i}\right)$ of the principal series satisfying
\[
    \tilde{f}_{\underline{\nu}}(umk) = f_{\underline{\nu}}(m)|m|^{\underline{\nu}}\tilde{f}_{\underline{\nu}}(k),
\]
where $|m|^{\nu} = |\det m_1|^{\nu_1}\ldots |\det m_n|^{\nu_n}$ for any $m\in M_{[a^n]}$. The Eisenstein series
\[
    E_n(f_\nu,\underline{\nu})(g) = \sum_{\gamma\in P_n(k)\backslash G(k)} \tilde{f}_{\underline{\nu}}(\gamma g)
\]
defines an intertwining operator between $G_n$ representations:
\begin{equation}\label{eis}
    E_n(\cdot,\underline{\nu}): I(\tau,\underline{\nu}):=\Ind_{P_n}^{G_{n}}\left(\bigotimes_{k=1}^{n}\tau|\cdot|^{\nu_i}\right) \longrightarrow\mathcal{A}_{G_{n}},
\end{equation}
and can be meromorphically continued to the whole $\mathbb{C}^n$ from the domain of convergence $\mathrm{Re}(\nu_i-\nu_j) > 0$ for $i<j$.

\subsection{The Rankin-Selberg \texorpdfstring{$L$}{L}-Function}
The Langlands-Shahidi method can be applied to define the Rankin-Selberg $L$-function $L(s,\tau\times\hat\tau)$ described in \cite{jacquet1983rankin} for any cuspidal automorphic representation $\tau$ as normalization factors of intertwining operators. An exposition of such method can be found in \cite[Chapter 5]{shahidi2010eisenstein}. We will be mostly following the setup in \cite[Appendice]{moeglin1989spectre}, but will also point out other properties of the Rankin-Selberg $L$-function $L(s,\tau\times\hat{\tau})$ which we will be using in this paper.\\

\indent We consider two irreducible unitary cuspidal automorphic representations $\tau_1,\tau_2$ on $GL(n_1,\mathbb{A})$ and $GL(n_2,\mathbb{A})$, respectively. The constant term of the Eisenstein series
\[
    E_{n_1,n_2}(\cdot, \underline{s}): \Ind_{P_{[n_1,n_2]}}^{G_{n_1+n_2}}(\tau_1|\cdot|^{\frac{s}{2}}\otimes\tau_2|\cdot|^{-\frac{s}{2}})\longrightarrow \mathcal{A}_{G_{n_1+n_2}}
\]
will have only one summand unless if $n_1=n_2$. In any case, following \cite{shahidi1983local}, the normalization factor $r(s,w_0)$ of the intertwining operator $M(s,w_0)$ can be expressed as
\[
    r(s,w_0) = \frac{L(s,\tau_1\times\hat{\tau}_2)}{L(1+s,\tau_1\times\tau_2)\epsilon(s,\psi,\tau_1\times\hat{\tau}_2)}
\]
where $L(s,\tau_1\times\hat{\tau}_2)$ is the completed Rankin-Selberg $L$-function with its local factors described in \cite{jacquet1983rankin}. In \cite[Appendice, Corollaire]{moeglin1989spectre}, it is shown that the function $L(s,\tau_1\times{\hat{\tau}_2})$ is entire unless $n_1=n_2$ and there exists a complex number $t$ such that $\tau_1 = \tau_2|\cdot|^{t}$. The poles of the function $L(s,\tau\times\hat{\tau})$ are simple and are located at $s=0,1$. A zero-free region of $L(s,\tau\times \hat{\tau})$ is given in \cite{goldfeld2018standard} by
\[
    \Re(s) > 1-\frac{c}{(\log(|\Im(s)|+2))^5}.
\]
A consequence of the zero-free region given above is that there are no zeros in the region $\Re(s)\geq 1$ and $\Re(s)\leq 0$. Therefore, one expects no cancellations of poles between the numerators and the denominators apart from the cancellations from the zeros on the critical strip.

\subsection{Residues of Eisenstein Series and Speh Representation}
\noindent In this section, we summarize the key facts concerning the poles and residues of the Eisenstein series $E_n(\cdot,\nu)$ defined in (\ref{eis}) from the previous section. By the Langlands constant term formula \cite[II.1.7]{moeglin1995spectral} and the cuspidality of $\tau$, the constant term along the unipotent radical of $U_n$ of the parabolic subgroup $P_n$ can be expressed as
\[
    c_{U_n} E_n(f,\underline{\nu})(g) = \sum_{w\in W_{a}^{\times n}\backslash W_{na}\cong \mathbb{S}_n} M(w,\underline{\nu}) \tilde{f}(g).
\]
Each intertwining operator $M(w,\underline{\nu})$ is defined in the dominant cone $\mathrm{Re}(\nu_i-\nu_j) \gg 0$, where $i<j$, as the following formal integral
\[
    M(w,\underline{\nu}) \tilde{f}(g) = \int_{U_n\cap w^{-1}U_n w\backslash U_n}\tilde{f}(wng)dn.
\]
The operator $M(w,\underline{\nu})$ can be meromorphically continued to the whole $\mathbb{C}^n$. By \cite[Appendice]{moeglin1989spectre}, each summand $M(w,\underline{\nu}) \tilde{f}(g)$ has the same set of poles as the function $r(w,\underline{\nu})$ known as the \emph{normalization factor}:
\[
    r(w,\underline{\nu}) = \prod_{\substack{i < j\\w(i)>w(j)}} \frac{L(\nu_i-\nu_j,\tau\times\hat{\tau})}{L(1+\nu_i-\nu_j,\tau\times\hat{\tau})\epsilon(\nu_i-\nu_j,\tau\times\hat{\tau},\psi)}.
\]
The normalized intertwining operator $N(w,\underline{\nu}) = r(w,\underline{\nu})^{-1} M(w,\underline{\nu})$ is holomorphic in the region $\mathrm{Re}(\nu_i-\nu_j)\geq 0$ with $i<j$.\\

The pole of the highest possible codimension $r(w,\underline{\nu})$ occurs when $w(i)>w(j)$ for all $i<j$. In this case, the normalization factor $r(w,\underline{\nu})$ has a simple pole at the intersection of the hyperplanes $\nu_{i} - \nu_{i+1} = 1$, i.e. at the point $\underline{\nu} = -\underline{\lambda}_n = \left(\frac{n-1}{2},\frac{n-3}{2},\ldots, \frac{1-n}{2}\right)$ (pay attention to the minus sign). We can take the residue
\[
    \mathrm{Res}_{n} E_n(f,\underline{\nu}) = \lim_{\underline{\nu}\rightarrow -\underline{\lambda}_n} \prod_{i=1}^{n-1} (\nu_{i}-\nu_{i+1}-1) E_n(f,\underline{\nu}).
\]
of the Eisenstein series $E_n(f,\underline{\nu})$ at the point $-\underline{\lambda}_n$. The irreducible automorphic representation $\Delta(\tau,n)$ generated by $\mathrm{Res}_n E_n(f,\underline{\nu})$ is usually referred to as the \emph{Speh representation}, which constitutes the whole non-cuspidal discrete spectrum if we let $\tau$ run through all possible choices of cuspidal automorphic representations for all block sizes $a$.\\

The composition $\mathrm{Res}_{n} E_n(\cdot,\underline{\nu})$ of the Eisenstein series operator and the residue operator is a quotient operator from the principal series $\Ind_{P_n}^G\left(\bigotimes_{k=1}^n \tau|\cdot|^{\nu_i}\right)$ to $\mathcal{A}_{G_n}$. It factors through the embedding $\Delta(\tau,n)\hookrightarrow \mathcal{A}_{G_n}$. By the Langlands classification (\cite{langlands1989irreducible} and \cite{langlands2006functional}), on the pole $\underline{\nu}$ where a residue exists, the representation $\Delta(\tau,n)$ is also the image of the intertwining operator
\[
    M(w_0, \underline{\nu}):  I(\tau, \underline{\nu})\longrightarrow I(\tau, w_0\underline{\nu})
\]
for the longest element $w_0$ of the Weyl group $\mathbb{S}_n$. On each finite place, the local component of the residual representation $\Delta(\tau,n)$ is isomorphic to the Langlands quotient of the principal series $I(\tau_v,\underline{\nu})$.

\subsection{Induction from Speh Representations}
As in Theorem \ref{thma}, for $m\leq n$ and a irreducible cuspidal unitary automorphic representation $\tau$, we construct the induced representation
\[
    I(\tau,\underline{s}):=\mathrm{Ind}_{P_{[m,n]}}^{G_{m+n}}\left(\Delta(\tau,m)|\cdot|^{s_1}\boxtimes\Delta(\tau,n)|\cdot|^{s_2}\right)
\]
from two Speh representations $\Delta(\tau,m)$ and $\Delta(\tau,n)$ for the groups $G_m$ and $G_n$, respectively. We would like to develop a combinatorial method to understand the poles of the Eisenstein series
\[
    E^{m,n}(\cdot,\underline{s}): I(\tau,\underline{s})\longrightarrow \mathcal{A}_{G_n}
\]
on a section $f$ of $I(\tau,\underline{s})$. Since the Speh representation $\Delta(\tau,m)|\cdot|^{s_1}\boxtimes \Delta(\tau,n)|\cdot|^{s_2}$ can be realized as a subrepresentation of $I(m,\underline{\lambda_m})|\cdot|^{s_1}\boxtimes I(n,\underline{\lambda_n})|\cdot|^{s_2}$, any vector of $I(\tau,\underline{s})$ can be realized as a vector in the principal series
\[
    I(m,n,\underline{s}) = \mathrm{Ind}_{P_{[m,n]}}^{G_{m+n}}\left(I(m,\underline{\lambda_m})|\cdot|^{s_1}\boxtimes I(n,\underline{\lambda_n})|\cdot|^{s_2}\right).
\]
For any such vector $f$, the constant term along the unipotent radical $U_{m+n}$ can be expressed as the following sum of formal integrals over Weyl group cosets:
\begin{align*}
    (c_{U_{m+n}}E^{m,n}(f,\underline{s}))(g) &= \sum_{\gamma\in P_{[m,n]}(k)\backslash G_{m+n}(k)}\int_{U_{m+n}(k)\backslash U_{m+n}(\mathbb{A})} f(\gamma n g) dg
\end{align*}
By the Bruhat decompsition, $\gamma$ can be represented by product $w\begin{spmatrix}1&\beta\\0&1\end{spmatrix}$, and we can decompose $n$ into the product $n = n_1n_2$, such that $n_1\in U_{[m,n]}$ and $n_2\in U_m\times U_n$. Writting $n_1 = \begin{spmatrix}1&a\\0&1\end{spmatrix}$ and $n_2 = \begin{spmatrix}u_1&0\\0&u_2\end{spmatrix}$, we have
\begin{align*}
    \gamma n &= w\begin{spmatrix}1&a + \beta\\0&1\end{spmatrix}\begin{spmatrix}u_1&0\\0&u_2\end{spmatrix}\\
    &=w\begin{spmatrix}u_1&0\\0&u_2\end{spmatrix}\begin{spmatrix}u_1^{-1}&0\\0&u_2^{-1}\end{spmatrix}\begin{spmatrix}1&a + \beta\\0&1\end{spmatrix}\begin{spmatrix}u_1&0\\0&u_2\end{spmatrix}\\
    &=w\begin{spmatrix}u_1&0\\0&u_2\end{spmatrix}\begin{spmatrix}1&u_1^{-1}(a + \beta)u_2\\0&1\end{spmatrix}.
\end{align*}
For any section $f\in I(m,n,\underline{s})$, since
\[
    f(\gamma ng) = f\left(w\begin{spmatrix}u_1&0\\0&u_2\end{spmatrix}w^{-1}w\begin{spmatrix}1&u_1^{-1}(a + \beta)u_2\\0&1\end{spmatrix}g\right)
\]
and $f$ is invariant under the left regular action of $U_{m+n}$, we have
\[
    f(\gamma ng) = f\left(w\begin{spmatrix}1&u_1^{-1}(a + \beta)u_2\\0&1\end{spmatrix}g\right).
\]
Therefore, for the section $f\in I(m,n,\underline{s})$, the constant term $c_{U_{m+n}}E^{m,n}(f,\underline{s})$ can simply be written as the sum
\[
    c_{U_{m+n}}E^{m,n}(f,\underline{s}) = \sum_{w\in \mathbb{S}_m\times \mathbb{S}_n\backslash \mathbb{S}_{m+n}}\int_{U_{[m,n]}(\mathbb{A})\cap w^{-1}U_{[m,n]}(\mathbb{A}) w\backslash U_{[m,n]}(\mathbb{A})} f(wng)dn = \sum_{w\in \mathbb{S}_m\times \mathbb{S}_n\backslash \mathbb{S}_{m+n}}M(w,\underline{s})f.
\]
We can summarize this result into the following lemma:
\begin{lemma}\label{lemb}
    For any section $f\in I(m,n,\underline{s})\hookrightarrow I(\tau,\underline{s})$, the Eisenstein series $E^{m,n}(f,\underline{s})$ has the same set of poles with the constant term
    \[
        c_{U_{m+n}}E_{\underline{s}}^{m,n}(f) = \sum_{w\in \mathbb{S}_{m}\times \mathbb{S}_{n}\backslash \mathbb{S}_{m+n}} M(w,\underline{s})f.
    \]
\end{lemma}
We can further organize the sum in Lemma \ref{lemb} by grouping $M(w,\underline{s})f$ for different $w$'s based on which space $M(w,\underline{s})f$ lives in. Denoting 
\[
    \rho_{m,n}(\tau,\underline{s}) = \left(\bigotimes_{i=1}^m \tau|\cdot|^{\frac{1-m}{2}+i-1+s_1}\right)\otimes \left(\bigotimes_{j=1}^n \tau|\cdot|^{\frac{1-n}{2}+j-1+s_2}\right),
\]
the intertwining operator $M(w,\underline{s})f$ is a vector in the conjugation $(\rho_{m,n}(\tau,\underline{s}))^w$. We will denote by $\mathcal{O}_{\underline{s}}(w)$ the set of all Weyl group representatives $w'\in \mathbb{S}_m\times\mathbb{S}_n\backslash\mathbb{S}_{m+n}$ such that $(\rho_{m,n}(\tau,\underline{s}))^w\cong (\rho_{m,n}(\tau,\underline{s}))^{w'}$. Thus the sum in Lemma \ref{lemb} can be grouped into sums on orbits $\mathcal{O}_{\underline{s}}(w)$:
\begin{align}
    c_{U_{m+n}}E_{\underline{s}}^{m,n}(f) = \sum_{\mathcal{O}_{\underline{s}}(w)}\left(\sum_{w\in \mathcal{O}_{\underline{s}}(w)} M(w,\underline{s})f\right).
\end{align}

\section{Segments and Weyl Group Elements}
In this section, we classify the orbits of elements in $\mathbb{S}_m\times\mathbb{S}_n\backslash \mathbb{S}_{m+n}$ and describe the representatives of each orbit.
\subsection{A Double Coset}\label{doublecoset}
We define $W_{m,n}$ as the collection of representatives of the double coset 
\[
    (W_{m}\times W_{n})\backslash W_{m+n} /\{e\} = \mathbb{S}_m\times\mathbb{S}_n\backslash \mathbb{S}_{m+n},
\]
which has a one-to-one correspondence with the set of elements $w\in\mathbb{S}_{m+n}$ preserving the order in the strings $\{1,\ldots,m\}$ and $\{m+1,\ldots,m+n\}$, respectively. Such an element satisifies the property
\begin{align}
    w(1)<\ldots <w(m),\;w(m+1)< \ldots <w(m+n).\label{defprop}
\end{align}
The action of $w$ on the string $\{1,\ldots,m+n\}$ shuffles between the two strings $\{1,\ldots,m\}$ and $\{m+1,\ldots,m+n\}$. We denote the interlacing intervals originated from each one of these two strings by $v_i$ and $u_i$, respectively, as shown in the following diagram:
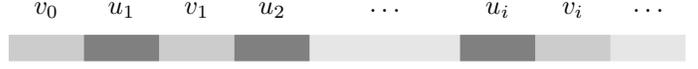
\begin{center}
    \begin{tikzpicture}[line width=10pt]
        \draw (0.5,0.5) node {$v_0$};
        \draw[opacity = 0.2] (0,0) --(1,0);
        \draw (1.5,0.5) node {$u_1$};
        \draw[opacity = 0.5] (1,0) --(2,0);
        \draw (2.5,0.5) node {$v_1$};
        \draw[opacity = 0.2] (2,0) --(3,0);
        \draw (3.5,0.5) node {$u_2$};
        \draw[opacity = 0.5] (3,0) --(4,0);
        \draw (5,0.5) node {$\ldots$};
        \draw[opacity = 0.1] (4,0) --(6,0);
        \draw (6.5,0.5) node {$u_i$};
        \draw[opacity = 0.5] (6,0) --(7,0);
        \draw (7.5,0.5) node {$v_i$};
        \draw[opacity = 0.2] (7,0) --(8,0);
        \draw (8.5,0.5) node {$\ldots$};
        \draw[opacity = 0.1] (8,0) --(9,0);
    \end{tikzpicture}.
\end{center}
The ``light gray'' $v_i$-intervals and the ``dark gray'' $u_i$-intervals assemble back to the two original strings:
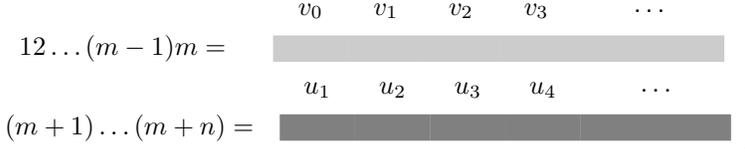
\begin{center}
    \begin{tikzpicture}[line width=10pt]
        \draw (-2,0) node {$12\ldots (m-1)m = $};
        \draw (0.5,0.5) node {$v_0$};
        \draw[opacity = 0.2] (0,0) --(1,0);
        \draw (1.5,0.5) node {$v_1$};
        \draw[opacity = 0.2] (1,0) --(2,0);
        \draw (2.5,0.5) node {$v_2$};
        \draw[opacity = 0.2] (2,0) --(3,0);
        \draw (3.5,0.5) node {$v_3$};
        \draw[opacity = 0.2] (3,0) --(4,0);
        \draw (5,0.5) node {$\ldots$};
        \draw[opacity = 0.2] (4,0) --(6,0);
    \end{tikzpicture}
    \begin{tikzpicture}[line width=10pt]
        \draw (-2,0) node {$(m+1)\ldots (m+n) = $};
        \draw (0.5,0.5) node {$u_1$};
        \draw[opacity = 0.5] (0,0) --(1,0);
        \draw (1.5,0.5) node {$u_2$};
        \draw[opacity = 0.5] (1,0) --(2,0);
        \draw (2.5,0.5) node {$u_3$};
        \draw[opacity = 0.5] (2,0) --(3,0);
        \draw (3.5,0.5) node {$u_4$};
        \draw[opacity = 0.5] (3,0) --(4,0);
        \draw (5,0.5) node {$\ldots$};
        \draw[opacity = 0.5] (4,0) --(6,0);
    \end{tikzpicture}.
\end{center}
It is possible that $v_0$ is empty, in which case the sequence $w(1\ldots (m+n))$ will start with the interval $u_1$:
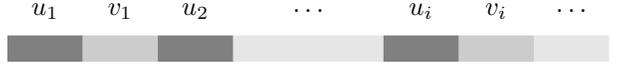
\begin{center}
    \begin{tikzpicture}[line width=10pt]
        \draw (1.5,0.5) node {$u_1$};
        \draw[opacity = 0.5] (1,0) --(2,0);
        \draw (2.5,0.5) node {$v_1$};
        \draw[opacity = 0.2] (2,0) --(3,0);
        \draw (3.5,0.5) node {$u_2$};
        \draw[opacity = 0.5] (3,0) --(4,0);
        \draw (5,0.5) node {$\ldots$};
        \draw[opacity = 0.1] (4,0) --(6,0);
        \draw (6.5,0.5) node {$u_i$};
        \draw[opacity = 0.5] (6,0) --(7,0);
        \draw (7.5,0.5) node {$v_i$};
        \draw[opacity = 0.2] (7,0) --(8,0);
        \draw (8.5,0.5) node {$\ldots$};
        \draw[opacity = 0.1] (8,0) --(9,0);
    \end{tikzpicture}.
\end{center}
For any interval $I$, we denote by $l(I)$ the index of its starting point on the left and $r(I)$ the index of its ending point on the right. In the string corresponding to $w$, if an interval $I_1$ appears to the left of $I_2$, then we say $I_1$ \emph{precedes} $I_2$, and we denote this situation by $I_1\prec I_2$.

\subsection{Segments}\label{segmentssec}
The full principal series induced from the tensor product $\rho_{m,n}(\tau,\underline{s})$ can be perceived as the automorphic analogue of the induction from \emph{Berstein-Zelevinsky segments} $\Delta + s_1$ and $\Delta'+s_2$ (c.f. \cite{zelevinsky1980induced}):
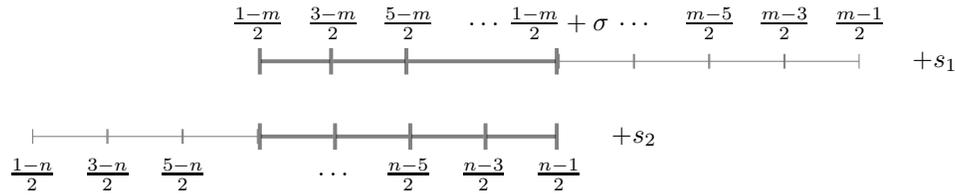
\begin{center}
    \begin{tikzpicture}
        \draw (0,0.5) node {$\frac{1-m}{2}$};
        \draw (1,0.5) node {$\frac{3-m}{2}$};
        \draw (2,0.5) node {$\frac{5-m}{2}$};
        \draw (3,0.5) node {$\ldots$};
        \draw (4,0.5) node {$\frac{1-m}{2}+\sigma$};
        \draw (5,0.5) node {$\ldots$};
        \draw (6,0.5) node {$\frac{m-5}{2}$};
        \draw (7,0.5) node {$\frac{m-3}{2}$};
        \draw (8,0.5) node {$\frac{m-1}{2}$};
        \draw[opacity = 0.5,|-|,line width=1.5pt] (0,0) --(1,0);
        \draw[opacity = 0.5,-|,line width=1.5pt] (1,0) --(2,0);
        \draw[opacity = 0.5,-|,line width=1.5pt] (2,0) --(4,0);
        \draw[opacity = 0.5,|-|] (4,0) --(5,0);
        \draw[opacity = 0.5,|-|] (5,0) --(6,0);
        \draw[opacity = 0.5,|-|] (6,0) --(7,0);
        \draw[opacity = 0.5,|-|] (7,0) --(8,0);
        \draw (9,0) node {$+s_1$};
        \draw (-3,-1.5) node {$\frac{1-n}{2}$};
        \draw (-2,-1.5) node {$\frac{3-n}{2}$};
        \draw (-1,-1.5) node {$\frac{5-n}{2}$};
        \draw (1,-1.5) node {$\ldots$};
        \draw (2,-1.5) node {$\frac{n-5}{2}$};
        \draw (3,-1.5) node {$\frac{n-3}{2}$};
        \draw (4,-1.5) node {$\frac{n-1}{2}$};
        \draw[opacity = 0.5,|-|] (-3,-1) --(-2,-1);
        \draw[opacity = 0.5,|-|] (-2,-1) --(-1,-1);
        \draw[opacity = 0.5,|-|] (-1,-1) --(0,-1);
        \draw[opacity = 0.5,|-,line width=1.5pt] (0,-1) --(1,-1);
        \draw[opacity = 0.5,|-,line width=1.5pt] (1,-1) --(2,-1);
        \draw[opacity = 0.5,|-,line width=1.5pt] (2,-1) --(3,-1);
        \draw[opacity = 0.5,|-|,line width=1.5pt] (3,-1) --(4,-1);
        \draw (5,-1) node {$+s_2$};
    \end{tikzpicture}
\end{center}
The integer $\sigma\geq 0$ equals to the length of the overlapping part of the two segments. The smallest possible $\sigma$ is equal to 0, in which case the two segments are \emph{juxtaposed}, as shown in the following diagram:
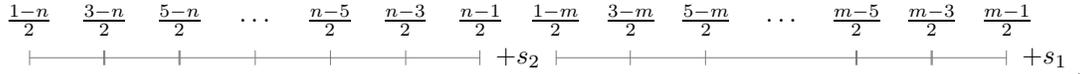
\begin{center}
    \begin{tikzpicture}
        \draw (0,0.5) node {$\frac{1-m}{2}$};
        \draw (1,0.5) node {$\frac{3-m}{2}$};
        \draw (2,0.5) node {$\frac{5-m}{2}$};
        \draw (3,0.5) node {$\ldots$};
        \draw (4,0.5) node {$\frac{m-5}{2}$};
        \draw (5,0.5) node {$\frac{m-3}{2}$};
        \draw (6,0.5) node {$\frac{m-1}{2}$};
        \draw[opacity = 0.5,|-|] (0,0) --(1,0);
        \draw[opacity = 0.5,-|] (1,0) --(2,0);
        \draw[opacity = 0.5,-|] (2,0) --(4,0);
        \draw[opacity = 0.5,|-|] (4,0) --(5,0);
        \draw[opacity = 0.5,|-|] (5,0) --(6,0);
        \draw (6.5,0) node {$+s_1$};
        \draw (-7,0.5) node {$\frac{1-n}{2}$};
        \draw (-6,0.5) node {$\frac{3-n}{2}$};
        \draw (-5,0.5) node {$\frac{5-n}{2}$};
        \draw (-4,0.5) node {$\ldots$};
        \draw (-3,0.5) node {$\frac{n-5}{2}$};
        \draw (-2,0.5) node {$\frac{n-3}{2}$};
        \draw (-1,0.5) node {$\frac{n-1}{2}$};
        \draw[opacity = 0.5,|-|] (-7,0) --(-6,0);
        \draw[opacity = 0.5,|-|] (-6,0) --(-5,0);
        \draw[opacity = 0.5,|-] (-5,0) --(-4,0);
        \draw[opacity = 0.5,|-] (-4,0) --(-3,0);
        \draw[opacity = 0.5,|-] (-3,0) --(-2,0);
        \draw[opacity = 0.5,|-|] (-2,0) --(-1,0);
        \draw (-0.5,0) node {$+s_2$};
    \end{tikzpicture}.
\end{center}
In general, if the length of the overlapping part of the two segments is $\sigma$,  we have:
\[
    \frac{n-1}{2}+s_2 = \frac{1-m}{2} + s_1 - 1 + \sigma,
\]
and thus $\sigma$ determines $s = s_1-s_2 = \frac{m+n}{2}-\sigma$, where $\sigma$ is allowed to be any integer between $0$ and $\lfloor\frac{m+n}{2}\rfloor$, so that the segments $\Delta+s_1$ and $\Delta'+s_2$ intersect in the following three ways:
\begin{itemize}
    \item Case I: $\sigma\leq \min\{m,n\}$:
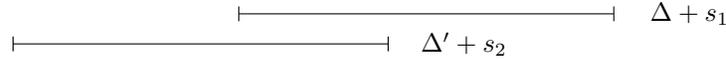
    \begin{center}
        \begin{tikzpicture}
            \draw[opacity = 1,|-|] (0,0.2) --(5,0.2);
            \draw (6,0.2) node {$\Delta+s_1$};
            \draw[opacity = 1,|-|] (-3,-0.2) --(2,-0.2);
            \draw (3,-0.2) node {$\Delta'+s_2$};
        \end{tikzpicture}
    \end{center}
    \item Case II: $m\leq\sigma\leq n$:
    
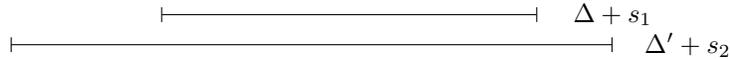
\begin{center}
        \begin{tikzpicture}
            \draw[opacity = 1,|-|] (0,0.2) --(5,0.2);
            \draw (6,0.2) node {$\Delta+s_1$};
            \draw[opacity = 1,|-|] (-2,-0.2) --(6,-0.2);
            \draw (7,-0.2) node {$\Delta'+s_2$};
        \end{tikzpicture}
    \end{center}
    \item Case III: $n \leq \sigma \leq m$:
    
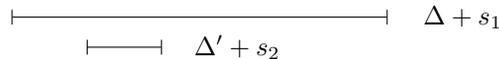
\begin{center}
        \begin{tikzpicture}
            \draw[opacity = 1,|-|] (0,0.2) --(5,0.2);
            \draw (6,0.2) node {$\Delta+s_1$};
            \draw[opacity = 1,|-|] (1,-0.2) --(2,-0.2);
            \draw (3,-0.2) node {$\Delta'+s_2$};
        \end{tikzpicture}
    \end{center}
\end{itemize}
The case when $\sigma > \max\{m,n\}$ is called \emph{$\Delta+s_1$ precedes $\Delta'+s_2$}, which will result in a negative $s=s_1-s_2$ and will not be considered.\\

Since the representation $\rho_{m,n}(\tau,\underline{s})$ depends on the choice of $\sigma$, the orbits $\mathcal{O}_{\underline{s}}(w)$ of Weyl group cosets $\mathbb{S}_m\times\mathbb{S}_n\backslash \mathbb{S}_{m+n}$ also depend on the choice of $\sigma$.

\subsection{Orbits}\label{orbits}
This section generalizes the method developed in \cite{hanzer2015images}. The condition for two Weyl group representatives $w,w'\in W_{m,n}$ belonging to the same orbit is $(\rho_{m,n}(\tau,\underline{s}))^w = (\rho_{m,n}(\tau,\underline{s}))^{w'}$, which is equivalent to requiring
\[
    w'(i) = w(i)\text{ or } w(m+n-\sigma+i).
\]
for $i\in \{1,\ldots,m\}$ and $m+n-\sigma+i\in \{m+1,\ldots,m+n\}$. Therefore, there is a transitive action by a group $G_{\sigma} = \ZZ_2^r$ (which is related to the $R$-group, and will be described in the following sections) of \emph{permissible moves} for some integer $r$ on each orbit $\mathcal{O}_{\underline{s}}$ permuting the pairs $\{w(i),w(m+n-\sigma+i)\}$. For any choice of $\sigma$, we can choose the \emph{base point} $w_0\in \mathcal{O}_{\underline{s}}$ to be the unique element in the orbit satisfying
\begin{equation}\label{req}
    w_0(i) < w_0(m+n-\sigma+i)
\end{equation}
for all $i$ such that $i\in \{1,\ldots,m\}$ and $m+n-\sigma+i\in \{m+1,\ldots,m+n\}$.

\subsection{Permissible Moves on Basepoints}\label{permissiblemoves}
In this section, we describe the aforementioned group $G_{\sigma}$ of \emph{permissible} moves.
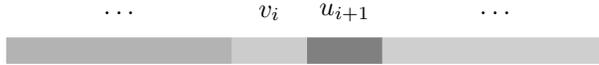
\begin{center}
    \begin{tikzpicture}[line width=10pt]
        \draw (4.5,0.5) node {$v_i$};
        \draw[opacity = 0.2] (4,0) --(5,0);
        \draw (5.5,0.5) node {$u_{i+1}$};
        \draw[opacity = 0.5] (5,0) --(6,0);
        \draw (7.5,0.5) node {$\ldots$};
        \draw[opacity = 0.1] (6,0) --(9,0);
        \draw (2.5,0.5) node {$\ldots$};
        \draw[opacity = 0.1] (6,0) --(9,0);
        \draw[opacity = 0.3] (1,0) --(4,0);
    \end{tikzpicture}
\end{center}
\begin{definition}\label{permissible}
    Fixing an $s=s_1-s_2=\frac{m+n}{2}-\sigma$, a \emph{permissible move} $g\in G_{\sigma}$ is a permutation of the string $\{1,2,\ldots,(m+n)\}$, such that
    \begin{enumerate}
        \item For any element $w$ satisfying (\ref{defprop}), $gw$ also satisfies (\ref{defprop}).
        \item The two elements $w$ and $gw$ satisfy $\rho_{m,n}(\tau,\underline{s})^w = \rho_{m,n}(\tau,\underline{s})^{gw}$, i.e.
        \[
            gw(i) \in\{ w(i),w(m+n-\sigma+i)\}.
        \]
    \end{enumerate} 
    Any element that can be affected by a permissible move is said to be \emph{alive}, otherwise it is said to be \emph{dead}. If there is a continuous segment of living elements, we refer to that segment as a \emph{living segment}, and vice versa for \emph{dead segments}.
\end{definition}

For a basepoint $w_0$, we will color the \emph{living} elements of the string $w_0(12\ldots(m+n))$ either green or red based on whether they originate from the string $12\ldots m$ or $(m+1)\ldots (m+n)$, respectively. Recall that we require the basepoint element $w_0$ to satisfy the property $w_0(i) < w_0(m+n-\sigma+i)$, on the basepoint element, we will color $w_0(i)$ green and $w_0(m+n-\sigma+i)$ red. We will call them \emph{green elements (intervals,resp.)} and \emph{red elements (intervals,resp.)} for convenience. Note that any green element lies on a ``light gray'' $v_i$-interval while any red element lies on a ``dark gray'' $u_i$-interval.  We will use this diagram to prove properties of all permissible moves from a basepoint $w_0$.
\begin{lemma}\label{lemma3}
    For a basepoint $w_0$, the following properties are true for green elements:
    \begin{enumerate}
        \item All elements between two green elements on the same ``light gray'' $v_i$-interval are green.
        
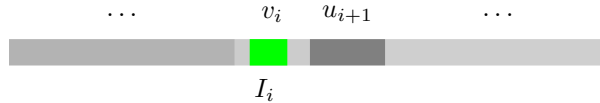
\begin{center}
            \begin{tikzpicture}[line width=10pt]
                \draw (4.5,0.5) node {$v_i$};
                \draw (4.4,-0.5) node {$I_i$};
                \draw[opacity = 0.2] (4,0) --(5,0);
                \draw[color = green] (4.2,0) --(4.7,0);
                \draw (5.5,0.5) node {$u_{i+1}$};
                \draw[opacity = 0.5] (5,0) --(6,0);
                \draw (7.5,0.5) node {$\ldots$};
                \draw[opacity = 0.1] (6,0) --(9,0);
                \draw (2.5,0.5) node {$\ldots$};
                \draw[opacity = 0.1] (6,0) --(9,0);
                \draw[opacity = 0.3] (1,0) --(4,0);
            \end{tikzpicture}
        \end{center}
        \item The right endpoint of $v_i$ and the right endpoint of the corresponding $I_i$ are aligned: $w(r(v_i)) = w(r(I_i))$.
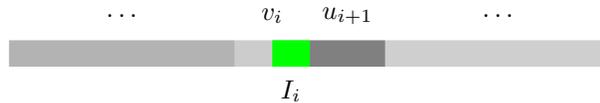
        \begin{center}
            \begin{tikzpicture}[line width=10pt]
                \draw (4.5,0.5) node {$v_i$};
                \draw (4.75,-0.5) node {$I_i$};
                \draw[opacity = 0.2] (4,0) --(5,0);
                \draw[color = green] (4.5,0) --(5,0);
                \draw (5.5,0.5) node {$u_{i+1}$};
                \draw[opacity = 0.5] (5,0) --(6,0);
                \draw (7.5,0.5) node {$\ldots$};
                \draw[opacity = 0.1] (6,0) --(9,0);
                \draw (2.5,0.5) node {$\ldots$};
                \draw[opacity = 0.1] (6,0) --(9,0);
                \draw[opacity = 0.3] (1,0) --(4,0);
            \end{tikzpicture}
        \end{center}
    \end{enumerate}
    Similarly, for red elements, they form a gap-free interval on any $u_i$, and their left endpoints are aligned with the left endpoint of $u_i$ they lie in:
    
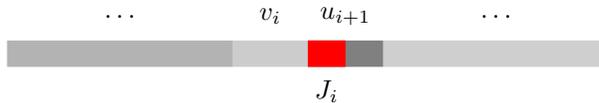
\begin{center}
        \begin{tikzpicture}[line width=10pt]
            \draw (4.5,0.5) node {$v_i$};
            \draw (5.25,-0.5) node {$J_i$};
            \draw[opacity = 0.2] (4,0) --(5,0);
            \draw (5.5,0.5) node {$u_{i+1}$};
            \draw[opacity = 0.5] (5,0) --(6,0);
            \draw (7.5,0.5) node {$\ldots$};
            \draw[opacity = 0.1] (6,0) --(9,0);
            \draw (2.5,0.5) node {$\ldots$};
            \draw[opacity = 0.1] (6,0) --(9,0);
            \draw[opacity = 0.3] (1,0) --(4,0);
            \draw[color = red] (5,0) --(5.5,0);
        \end{tikzpicture}.
    \end{center}
\end{lemma}
\begin{proof}
    These properties are all consequences of (1) in Definition \ref{permissible}. On any $v$-interval, the green elements are not allowed to move beyond any dead interval to its right:
    
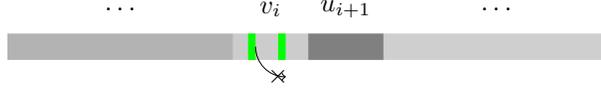
\begin{center}
        \begin{tikzpicture}[line width=10pt]
            \draw (4.5,0.5) node {$v_i$};
            \draw[opacity = 0.2] (4,0) --(5,0);
            \draw[color = green] (4.2,0) --(4.3,0);
            \draw[color = green] (4.6,0) --(4.7,0);
            \draw [->, very thin] (4.3,0) arc (180:270:0.4);
            \draw (4.6,-0.4) node {$\times$};
            \draw (5.5,0.5) node {$u_{i+1}$};
            \draw[opacity = 0.5] (5,0) --(6,0);
            \draw (7.5,0.5) node {$\ldots$};
            \draw[opacity = 0.1] (6,0) --(9,0);
            \draw (2.5,0.5) node {$\ldots$};
            \draw[opacity = 0.1] (6,0) --(9,0);
            \draw[opacity = 0.3] (1,0) --(4,0);
        \end{tikzpicture}.
    \end{center}
    Therefore, on each $v$-interval, no gap is allowed on any of its green subintervals. For the same reason, there are no dead elements to its right. Therefore, the right end of $v_i$ and the right end of $I_i$ are aligned.
\end{proof}
\begin{corollary}
    On the same $v_i$, the green elements form a gap-free interval, and on each ``light gray'' interval $v_i$ there is only one interval of green elements. We denote this interval by $I_i$. Similarly, on the same $u_i$, the red elements form a unique gap-free interval, and we denote this interval by $J_i$.
\end{corollary}

The following lemma describes the permissible moves near the dead elements:

\begin{lemma}
    On the string representing the basepoint $w_0$, the following statements are true:
    \begin{enumerate}
        \item If the element with index $w(r(I_i))+1$ lying to the right of a green interval $I_i$ is dead, the permissible moves can only move $I_i$ into the adjacent red interval in $u_{i+1}$, with their right endpoints aligned: 
        
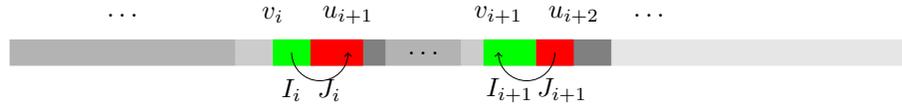
\begin{center}
            \begin{tikzpicture}[line width=10pt]
                \draw (4.5,0.5) node {$v_i$};
                \draw (4.75,-0.5) node {$I_i$};
                \draw (5.25,-0.5) node {$J_i$};
                \draw (7.65,-0.5) node {$I_{i+1}$};
                \draw (8.35,-0.5) node {$J_{i+1}$};
                \draw[opacity = 0.2] (4,0) --(5,0);
                \draw (5.5,0.5) node {$u_{i+1}$};
                \draw[opacity = 0.5] (5,0) --(6,0);
                \draw[opacity = 0.3] (6,0) --(7,0);
                \draw (7.5,0.5) node {$v_{i+1}$};
                \draw[opacity = 0.2] (7,0) --(8,0);
                \draw (8.5,0.5) node {$u_{i+2}$};
                \draw[opacity = 0.5] (8,0) --(9,0);
                \draw (9.5,0.5) node {$\ldots$};
                \draw[opacity = 0.1] (9,0) --(13,0);
                \draw (2.5,0.5) node {$\ldots$};
                \draw[opacity = 0.3] (1,0) --(4,0);
                \draw[color = green] (4.5,0) --(5,0);
                \draw[color = red] (5,0) --(5.7,0);
                \draw[color = green] (7.3,0) --(8,0);
                \draw[color = red] (8,0) --(8.5,0);
                \draw (6.5,0) node {$\ldots$};
                \draw [->, very thin] (4.75,0) arc (180:360:0.375);
                \draw [->, very thin] (8.25,0) arc (360:180:0.375);
            \end{tikzpicture}.
        \end{center}
        Similarly, if the element with index $w(l(J_{i+1})) - 1$ to the left of $J_{i+1}\subset u_{i+2}$ is dead, the permissible moves can only move $J_{i+1}$ leftwards to its adjacent green interval $I_{i+1}$ in $v_{i+1}$, with their left endpoints aligned.
        \item If the intervals $I_i \subset v_i$ and $I_{i+1}\subset v_{i+1}$ are connected (i.e. without any gap in between), then the whole $u_{i+1}$ is a red interval $J_i$.
        % \item A green interval always ends with some $I$ such that the whole interval $v_{i}$ is a green interval $I_i$ on the right,  and a red interval always ends with some $J$ such that the whole interval $u_{i}$ is a red interval on the left. Between the ends of a connected green interval, all green intervals across different $v_i$'s must move forward as a whole.
    \end{enumerate}
\end{lemma}
\begin{proof}
    The first part is a consequence of (1) in Definition \ref{permissible} and Lemma \ref{lemma3} since there is no green living interval adjacent to $I_i$ on the right. The second part is a consequence of the first part and (2) in Definition \ref{permissible}.
\end{proof}
\begin{remark}
    There is no specification about how on any basepoint element $w_0$, the elements are moved between continuous green/red intervals. One green interval $I_i$ can be splitted across two $u$-intervals, and one red interval $J_i$ can be splitted across two $v$-intervals by a permissible move. The only fact we know for sure is that, the left endpoint of a continuous interval satisfies the same property as $I_{i+1}$ and $J_{i+1}$ in the Part (2) of the lemma above, and the right endpoint of a continuous interval satisfies the same property as $I_{i}$ and $J_{i}$ in the Part (1) of the lemma above.
\end{remark}

The following lemma describes the gaps between two continuous intervals. In fact, this lemma implies that there is actually no gap between two consecutive living intervals.

\begin{lemma}\label{lemmamove}
    \begin{enumerate}
        \item For the following configuration on four consecutive intervals $v_i, u_{i+1}, v_{i+1}, u_{i+2}$:
        
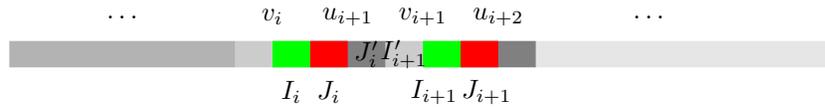
\begin{center}
            \begin{tikzpicture}[line width=10pt]
                \draw (4.5,0.5) node {$v_i$};
                \draw (4.75,-0.5) node {$I_i$};
                \draw (5.25,-0.5) node {$J_i$};
                \draw (6.65,-0.5) node {$I_{i+1}$};
                \draw (7.35,-0.5) node {$J_{i+1}$};
                \draw[opacity = 0.2] (4,0) --(5,0);
                \draw (5.5,0.5) node {$u_{i+1}$};
                \draw[opacity = 0.5] (5,0) --(6,0);
                \draw (6.5,0.5) node {$v_{i+1}$};
                \draw[opacity = 0.2] (6,0) --(7,0);
                \draw (7.5,0.5) node {$u_{i+2}$};
                \draw[opacity = 0.5] (7,0) --(8,0);
                \draw (9.5,0.5) node {$\ldots$};
                \draw[opacity = 0.1] (8,0) --(12,0);
                \draw (2.5,0.5) node {$\ldots$};
                \draw[opacity = 0.3] (1,0) --(4,0);
                \draw[color = green] (4.5,0) --(5,0);
                \draw[color = red] (5,0) --(5.5,0);
                \draw[color = green] (6.5,0) --(7,0);
                \draw[color = red] (7,0) --(7.5,0);
                \draw (5.75,0) node {$J_i'$};
                \draw (6.25,0) node {$I_{i+1}'$};
            \end{tikzpicture}
        \end{center}
        Setting $J_{i}' = u_{i+1}\backslash J_i$ and $I_{i+1}' = v_{i+1}\backslash I_{i+1}$, swapping $I'_{i+1}$ and $J_i'$ is a permissible move.
        \item For the following configuration,
        
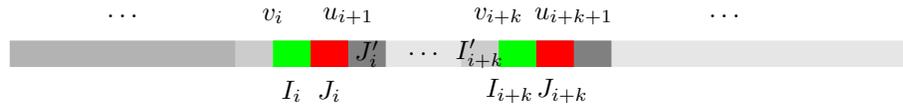
\begin{center}
            \begin{tikzpicture}[line width=10pt]
                \draw (4.5,0.5) node {$v_i$};
                \draw (4.75,-0.5) node {$I_i$};
                \draw (5.25,-0.5) node {$J_i$};
                \draw (7.65,-0.5) node {$I_{i+k}$};
                \draw (8.35,-0.5) node {$J_{i+k}$};
                \draw[opacity = 0.2] (4,0) --(5,0);
                \draw (5.5,0.5) node {$u_{i+1}$};
                \draw[opacity = 0.5] (5,0) --(6,0);
                \draw (7.5,0.5) node {$v_{i+k}$};
                \draw[opacity = 0.2] (7,0) --(8,0);
                \draw (8.5,0.5) node {$u_{i+k+1}$};
                \draw[opacity = 0.5] (8,0) --(9,0);
                \draw (10.5,0.5) node {$\ldots$};
                \draw[opacity = 0.1] (9,0) --(13,0);
                \draw (2.5,0.5) node {$\ldots$};
                \draw[opacity = 0.3] (1,0) --(4,0);
                \draw[color = green] (4.5,0) --(5,0);
                \draw[color = red] (5,0) --(5.5,0);
                \draw[color = green] (7.5,0) --(8,0);
                \draw[color = red] (8,0) --(8.5,0);
                \draw (5.75,0) node {$J_i'$};
                \draw (7.25,0) node {$I_{i+k}'$};
                \draw[opacity = 0.1] (6,0) --(7,0);
                \draw (6.5,0) node {$\ldots$};
            \end{tikzpicture}
        \end{center}
        one can construct a permissible move for the whole interval between $r(I_{i})+1$ and $l(I_{i+k})-1$.
    \end{enumerate}
\end{lemma}
\begin{proof}
    \begin{enumerate}
        \item Since $J_i'$ has the same length as $I_{i+1}'$, swapping $J_i'$ with $I_{i+1}'$ is in fact a permissible move. 
        \item We can list the intervals between $J_i$ and $I_{i+k}$ as
        \[
            J_i', v_{i+1}, u_{i+2},\ldots, u_{i+k},I'_{i+k}.
        \]
        Since the right endpoints of $I_i$ and $J_i$ correspond, and the left endpoints of $I_{i+k}$ and $J_{i+k}$ correspond, denoting the lengths of the ``dark grey'' intervals $J_i',u_{i+2},\ldots,u_{i+k}$ by $p_1,\ldots,p_k$, and the lengths of ``light grey'' intervals $v_{i+1},\ldots,I_{i+k}'$ by $q_1,\ldots,q_k$. We have $\sum_{i=1}^k p_i = \sum_{i=1}^k q_i$. From the partitions $(p_1,\ldots,p_k)$ and $(q_1,\ldots,q_k)$, we can obtain a refined partition of the dark gray and light gray intervals. The move between corresponding intervals is a permissible move.
    \end{enumerate}
\end{proof}
From the previous lemmas, we can find that on the segment $12\ldots m$, the green elements only exist along $(\max\{\sigma-n+1,1\},\ldots,\min\{\sigma,m\})$, and can be organized into mutually connected intervals $K_1,\ldots,K_r$, possibly straddling across different $v$-intervals. Each $K_t$ moves as a whole under the action of $G_{\sigma}$. The order of $G_{\sigma}$ is equal to $2^r$.

\subsection{General Description of the Permissible Moves}
Recall from Section \ref{doublecoset} that the string \[w(12\ldots(m+n))\] can be expressed as the interlacing $v,u$ intervals:
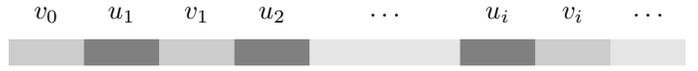
\begin{center}
    \begin{tikzpicture}[line width=10pt]
        \draw (0.5,0.5) node {$v_0$};
        \draw[opacity = 0.2] (0,0) --(1,0);
        \draw (1.5,0.5) node {$u_1$};
        \draw[opacity = 0.5] (1,0) --(2,0);
        \draw (2.5,0.5) node {$v_1$};
        \draw[opacity = 0.2] (2,0) --(3,0);
        \draw (3.5,0.5) node {$u_2$};
        \draw[opacity = 0.5] (3,0) --(4,0);
        \draw (5,0.5) node {$\ldots$};
        \draw[opacity = 0.1] (4,0) --(6,0);
        \draw (6.5,0.5) node {$u_i$};
        \draw[opacity = 0.5] (6,0) --(7,0);
        \draw (7.5,0.5) node {$v_i$};
        \draw[opacity = 0.2] (7,0) --(8,0);
        \draw (8.5,0.5) node {$\ldots$};
        \draw[opacity = 0.1] (8,0) --(9,0);
    \end{tikzpicture}.
\end{center}
We would like to use this diagram, combined with the lemmas we stated in the previous section to describe the green and red living intervals and the corresponding permissible moves.

\begin{proposition}\label{firstgreen}
    On the string representing a basepoint element $w_0$,
    \begin{enumerate}
        \item The first green element is the first green element on the left such that there is no dead element from $\{m+1,\ldots,m+n\}$ to its right. 
        \item The last red element is the last red element on the right such that there is no dead element from $\{1,\ldots,m\}$ to its left. %(Thus, the last red interval is a full interval $u_l$).
        \item An element $i\in \{1,\ldots,\sigma\}$ is dead if between $w(i)$ and $w(m+n-\sigma+i)$ there is a dead element coming from $\{m+n-\sigma+1,\ldots,m+n\}$.
    \end{enumerate}
\end{proposition}
\begin{proof}
    Part (1) and (2) follow from the definition of living elements and the proof of Lemma \ref{lemma3}. Part (3) is true simply because permissible moves do not allow any move across dead elements.
\end{proof}
% For any index $i$ on the string $w(1\ldots(m+n))$, denote by $M(i)$ the number of elements
% \[
%     M(i) = \#\left\{k\in 1\ldots \sigma\mid w(k)\prec i\right\}
% \]
% and $N(i)$ by the number of elements
% \[
%     N(i) = \#\left\{k\in (m+n-\sigma+1)\ldots (m+n)\mid w(k)\prec i\right\}
% \]
% \begin{proposition}
%     On a basepoint element $w_0$,
%     \begin{enumerate}
%         \item The dead elements $i\in \{1,\ldots,\sigma\}$ on the head are those elements who satisfy
%         \[
%             M(w(m+n-\sigma+j)) > N(w(m+n-\sigma+j))
%         \]
%         for all indices $j\leq i$.
%         \item The dead elements $i\in \{m+n-\sigma+1,\ldots,m+n\}$ on the tail are those elements who satisfy
%         \[
%             M(w(j)) < N(w(j))
%         \]
%         for all indices $j\geq i$.
%     \end{enumerate}
% \end{proposition}
% \begin{proof}
    
% \end{proof}
In the proof of Lemma \ref{lemmamove}, we have already seen that there is no dead element between the first and the last living elements. Denoting the first living green interval by $v_k'\subset v_k$, and the last living red interval by $u_l'\subset u_l$. The positions of these living elements are displayed as in the following diagram:
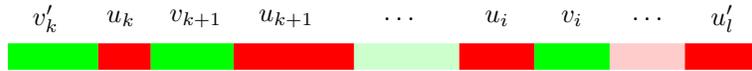
\begin{center}
    \begin{tikzpicture}[line width=10pt]
        \draw (0.5,0.5) node {$v_k'$};
        \draw[color = green] (0,0) --(1.2,0);
        \draw (1.5,0.5) node {$u_k$};
        \draw[color = red] (1.2,0) --(1.9,0);
        \draw (2.5,0.5) node {$v_{k+1}$};
        \draw[color = green] (1.9,0) --(3,0);
        \draw (3.7,0.5) node {$u_{k+1}$};
        \draw[color = red] (3,0) --(4.6,0);
        \draw[color = green, opacity = 0.2] (4.6,0) --(6,0);
        \draw (5.2,0.5) node {$\ldots$};
        \draw (6.5,0.5) node {$u_i$};
        \draw[color = red] (6,0) --(7,0);
        \draw (7.5,0.5) node {$v_i$};
        \draw[color = green] (7,0) --(8,0);
        \draw (8.5,0.5) node {$\ldots$};
        \draw[color = red, opacity = 0.2] (8,0) --(9,0);
        \draw (9.5,0.5) node {$u_l'$};
        \draw[color = red] (9,0) --(10,0);
    \end{tikzpicture}
\end{center}
The permissible moves can be constructed with the following algorithm:
\begin{enumerate}
    \item Denoting the lengths of the intervals $v_k',\ldots, v_l$ and $u_k,\ldots,u_l'$ by $p_k,\ldots,p_l$ and $q_k,\ldots,q_l$, for any $1 \leq r \leq \sum_{i=k}^l p_i = \sum_{i=k}^l q_i$, the collection of $S$ of ``end points of refined intervals'' $\{t_1,\ldots,t_N\}$ is defined as all the $r$'s such that $r$ corresponds to an end point of either a green or a red interval.
    
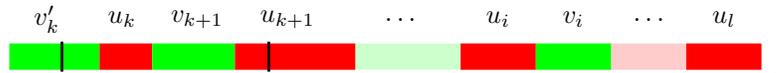
\begin{center}
        \begin{tikzpicture}[line width=10pt]
            \draw (0.5,0.5) node {$v_k'$};
            \draw[color = green] (0,0) --(1.2,0);
            \draw[color = black, line width=1pt] (0.7,0.2) --(0.7,-0.2);
            \draw (1.5,0.5) node {$u_k$};
            \draw[color = red] (1.2,0) --(1.9,0);
            \draw (2.5,0.5) node {$v_{k+1}$};
            \draw[color = green] (1.9,0) --(3,0);
            \draw (3.7,0.5) node {$u_{k+1}$};
            \draw[color = red] (3,0) --(4.6,0);
            \draw[color = black, line width=1pt] (3.45,0.2) --(3.45,-0.2);
            \draw[color = green, opacity = 0.2] (4.6,0) --(6,0);
            \draw (5.2,0.5) node {$\ldots$};
            \draw (6.5,0.5) node {$u_i$};
            \draw[color = red] (6,0) --(7,0);
            \draw (7.5,0.5) node {$v_i$};
            \draw[color = green] (7,0) --(8,0);
            \draw (8.5,0.5) node {$\ldots$};
            \draw[color = red, opacity = 0.2] (8,0) --(9,0);
            \draw (9.5,0.5) node {$u_l$};
            \draw[color = red] (9,0) --(10,0);
        \end{tikzpicture}.
    \end{center}
    \item Between these end points, we can obtain a refinement of the living green and red intervals. We label each interval of each color from left to right with an index:
    
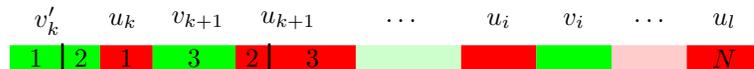
\begin{center}
        \begin{tikzpicture}[line width=10pt]
            \draw (0.5,0.5) node {$v_k'$};
            \draw[color = green] (0,0) --(1.2,0);
            \draw[color = black, line width=1pt] (0.7,0.2) --(0.7,-0.2);
            \draw (1.5,0.5) node {$u_k$};
            \draw[color = red] (1.2,0) --(1.9,0);
            \draw (2.5,0.5) node {$v_{k+1}$};
            \draw[color = green] (1.9,0) --(3,0);
            \draw (3.7,0.5) node {$u_{k+1}$};
            \draw[color = red] (3,0) --(4.6,0);
            \draw[color = black, line width=1pt] (3.45,0.2) --(3.45,-0.2);
            \draw[color = green, opacity = 0.2] (4.6,0) --(6,0);
            \draw (5.2,0.5) node {$\ldots$};
            \draw (6.5,0.5) node {$u_i$};
            \draw[color = red] (6,0) --(7,0);
            \draw (7.5,0.5) node {$v_i$};
            \draw[color = green] (7,0) --(8,0);
            \draw (8.5,0.5) node {$\ldots$};
            \draw[color = red, opacity = 0.2] (8,0) --(9,0);
            \draw (9.5,0.5) node {$u_l$};
            \draw[color = red] (9,0) --(10,0);
            \draw (0.35,0) node {1};
            \draw (0.95,0) node {2};
            \draw (1.55,0) node {1};
            \draw (2.45,0) node {3};
            \draw (3.225,0) node {2};
            \draw (4.025,0) node {3};
            \draw (9.5,0) node {$N$};
        \end{tikzpicture}
    \end{center}
    Following the convention of step 2, the labels of each color are $\{1,\ldots,N\}$. Each labeled interval with index $i$ is denoted by $\delta_i$.
    \item Now we can group the labeled intervals by collecting all the labels $k$ on the red intervals which satisfy the condition
    \[
        \max\{\text{labels of the green intervals to the left of }k\} = k.
    \]
    We denote the set by $Q$ and the labels by ${k_1,\ldots,k_M}$. 
\end{enumerate}
Thus we have proven the following proposition:
\begin{proposition}\label{segmentation}
    On an basepoint elemnt $w_0$, we can segment the green intervals into
    \begin{align*}
        K_1 &= (\Xi_1,\ldots,\Xi_{k_1})\\
        K_2 &= (\Xi_{k_1+1},\ldots,\Xi_{k_2})\\
        \ldots\\
        K_M &= (\Xi_{k_{M-1}+1},\ldots,\Xi_{k_M}).
    \end{align*}
    The any permissible move moves $K_i$ as a whole to the corresponding red interval $L_i$. Similar operations can be done for red intervals, and we can define $L_1,\ldots,L_M$ as the red intervals corresponding to each $K_i$.
\end{proposition}

\subsection{Dead Intervals}
In this and the following sections, we will call the dead intervals to the left of the living intervals the \emph{head}, and those dead intervals to the right of the living intervals the \emph{tail}. In this section, we describe the head and tail intervals in all the three cases described in Section \ref{segmentssec}.
\subsubsection{Head intervals for Cases I and II}
In the cases I and II, we have $n\geq \sigma$, thus:
\begin{proposition}\label{prophead}
    The \emph{head} can be constructed through the following procedure
    
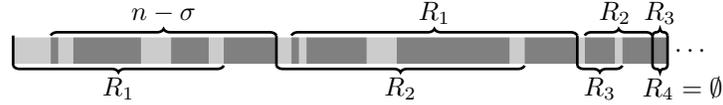
\begin{center}
        \begin{tikzpicture}[line width=10pt]
            \draw (1.5,0.5) node {$n-\sigma$};
            \draw (0.9,-0.5) node {$R_1$};
            \draw[color = black, line width=1pt] (-0.5,0.2) --(-0.5,-0.2);
            \draw[opacity = 0.5] (0,0) --(0.1,0);
            \draw[opacity = 0.2] (-0.5,0) --(0,0);
            \draw[opacity = 0.2] (0.1,0) --(0.3,0);
            \draw[opacity = 0.5] (0.3,0) --(1.2,0);
            \draw[opacity = 0.2] (1.2,0) --(1.6,0);
            \draw[opacity = 0.5] (1.6,0) --(2.1,0);
            \draw[opacity = 0.2] (2.1,0) --(2.3,0);
            \draw[opacity = 0.5] (2.3,0) --(3,0);
            \draw[color = black, line width=1pt] (3,0.2) --(3,-0.2);
            \draw (5.1,0.5) node {$R_1$};
            \draw (4.65,-0.5) node {$R_2$};
            \draw[opacity = 0.2] (3,0) --(3.2,0);
            \draw[opacity = 0.5] (3.2,0) --(3.3,0);
            \draw[opacity = 0.2] (3.3,0) --(3.4,0);
            \draw[opacity = 0.5] (3.4,0) --(4.2,0);
            \draw[opacity = 0.2] (4.2,0) --(4.6,0);
            \draw[opacity = 0.5] (4.6,0) --(6.1,0);
            \draw[opacity = 0.2] (6.1,0) --(6.3,0);
            \draw[opacity = 0.5] (6.3,0) --(7,0);
            \draw[color = black, line width=1pt] (7,0.2) --(7,-0.2);
            \draw (7.5,0.5) node {$R_2$};
            \draw (7.3,-0.5) node {$R_3$};
            \draw[opacity = 0.2] (7,0) --(7.1,0);
            \draw[opacity = 0.5] (7.1,0) --(7.5,0);
            \draw[opacity = 0.2] (7.5,0) --(7.6,0);
            \draw[opacity = 0.5] (7.6,0) --(8,0);
            \draw[color = black, line width=1pt] (8,0.2) --(8,-0.2);
            \draw (8.1,0.5) node {$R_3$};
            \draw[opacity = 0.5] (8,0) --(8.2,0);
            \draw[color = black, line width=1pt] (8.2,0.2) --(8.2,-0.2);
            \draw [decorate, line width = 1pt,decoration = {brace,mirror}] (-0.5,-0.2) --  (2.3,-0.2);
            \draw [decorate, line width = 1pt,decoration = {brace,mirror}] (3,-0.2) --  (6.3,-0.2);
            \draw [decorate, line width = 1pt,decoration = {brace,mirror}] (7,-0.2) --  (7.6,-0.2);
            \draw [decorate, line width = 1pt,decoration = {brace,mirror}] (8,-0.2) --  (8.2,-0.2);
            \draw (8.4,-0.5) node {$R_4=\emptyset$};
            \draw [decorate, line width = 1pt,decoration = {brace}] (0,0.2) --  (3,0.2);
            \draw [decorate, line width = 1pt,decoration = {brace}] (3.2,0.2) --  (7,0.2);
            \draw [decorate, line width = 1pt,decoration = {brace}] (7.1,0.2) --  (8,0.2);
            \draw [decorate, line width = 1pt,decoration = {brace}] (8,0.2) --  (8.2,0.2);
            \draw (8.5,0) node {$\ldots$};
        \end{tikzpicture}
    \end{center}
    Denote
    \begin{itemize}
        \item the set of elements $i$ in $\{1,\ldots,m\}$ such that $w(i)\prec w(m+n-\sigma)$ by $R_1$,
        \item the set of elements $i$ in $\{1,\ldots,m\}$ such that $w(|R_1|) \prec w(i)\prec w(m+n-\sigma+|R_1|)$ by $R_2$,
        \item the set of elements $i$ in $\{1,\ldots,m\}$ such that $w(|R_1|+|R_2|) \prec w(i)\prec w(m+n-\sigma+|R_1|+|R_2|)$ by $R_3$,
        \item $\ldots$
        \item the set of elements $i$ in $\{1,\ldots,m\}$ such that $w(\sum_{j=1}^{k-1}|R_j|) \prec w(i)\prec w(m+n-\sigma+\sum_{j=1}^{k-1}|R_j|)$ by $R_{k}$,
        \item The head ends where $R_{k+1} = \emptyset$. 
    \end{itemize}
    Then the set of all dead ``light gray'' elements in the head is the union of all $R_j$'s. 
\end{proposition}
\begin{proof}
    This theorem is a consequence of Proposition \ref{firstgreen}. There cannot be any dead element on the right of a living element. Therefore, all $i$'s such that $w(i)\prec w(m+n-\sigma)$ are dead, and so are all the elements $m+n-\sigma+i$ such that $i\in R_1$. As a consequence, all elements $i$ such that $w(i)\prec w(m+n-\sigma+|R_1|)$ are also dead. We can continue with this process, until any of the following three situations occur:
    \begin{enumerate}
        \item All ``light gray'' elements from $\{1,\ldots,m\}$ are finished.
        \item All ``dark gray'' elements from $\{m+1,\ldots,n\}$ are finished.
        \item $R_{k+1} = \emptyset$ for some $k$.
    \end{enumerate}
    If $(3)$ happens before $(1)$ or $(2)$ could happen, one will be able to construct a permissible move on the living intervals to the right of the element $\sum_{j=1}^{k}|R_j|+1$.
\end{proof}
\begin{lemma}\label{lemmahead}
If for any $1\leq j\leq k$ such that $R_j\neq \emptyset$, we have
\begin{itemize}
    \item $i \leq w(i) < n-\sigma+i$ for $i\in R_1$,
    \item $n-\sigma+i\leq w(i) < n-\sigma + |R_1|+i$ for $i\in R_2$,
    \item \ldots
    \item $n-\sigma + \sum_{j=1}^{k-2}|R_j|+i\leq w(i) < n-\sigma + \sum_{j=1}^{k-1}|R_j|+i$ for $i\in R_k$.
\end{itemize}
\end{lemma}
\begin{proof}
    The inequalities follows simply from the enumeration of intervals. The strict inequality is due to the nonemptyness of $R_j$.
\end{proof}
\subsubsection{Head intervals for Case III}
In Case III, we can construct the head following a similar process. In this case, we have $\sigma > n$, and the dead intervals can be obtained via the process described in the following proposition:
\begin{proposition}\label{prophead3}
    The \emph{head} can be constructed through the following procedure in Case III:
    
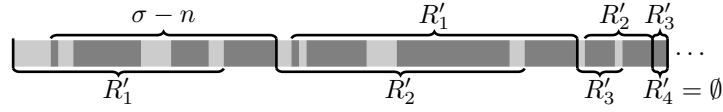
\begin{center}
        \begin{tikzpicture}[line width=10pt]
            \draw (1.5,0.5) node {$\sigma-n$};
            \draw (0.9,-0.5) node {$R_1'$};
            \draw[color = black, line width=1pt] (-0.5,0.2) --(-0.5,-0.2);
            \draw[opacity = 0.5] (0,0) --(0.1,0);
            \draw[opacity = 0.2] (-0.5,0) --(0,0);
            \draw[opacity = 0.2] (0.1,0) --(0.3,0);
            \draw[opacity = 0.5] (0.3,0) --(1.2,0);
            \draw[opacity = 0.2] (1.2,0) --(1.6,0);
            \draw[opacity = 0.5] (1.6,0) --(2.1,0);
            \draw[opacity = 0.2] (2.1,0) --(2.3,0);
            \draw[opacity = 0.5] (2.3,0) --(3,0);
            \draw[color = black, line width=1pt] (3,0.2) --(3,-0.2);
            \draw (5.1,0.5) node {$R_1'$};
            \draw (4.65,-0.5) node {$R_2'$};
            \draw[opacity = 0.2] (3,0) --(3.2,0);
            \draw[opacity = 0.5] (3.2,0) --(3.3,0);
            \draw[opacity = 0.2] (3.3,0) --(3.4,0);
            \draw[opacity = 0.5] (3.4,0) --(4.2,0);
            \draw[opacity = 0.2] (4.2,0) --(4.6,0);
            \draw[opacity = 0.5] (4.6,0) --(6.1,0);
            \draw[opacity = 0.2] (6.1,0) --(6.3,0);
            \draw[opacity = 0.5] (6.3,0) --(7,0);
            \draw[color = black, line width=1pt] (7,0.2) --(7,-0.2);
            \draw (7.5,0.5) node {$R_2'$};
            \draw (7.3,-0.5) node {$R_3'$};
            \draw[opacity = 0.2] (7,0) --(7.1,0);
            \draw[opacity = 0.5] (7.1,0) --(7.5,0);
            \draw[opacity = 0.2] (7.5,0) --(7.6,0);
            \draw[opacity = 0.5] (7.6,0) --(8,0);
            \draw[color = black, line width=1pt] (8,0.2) --(8,-0.2);
            \draw (8.1,0.5) node {$R_3'$};
            \draw[opacity = 0.5] (8,0) --(8.2,0);
            \draw[color = black, line width=1pt] (8.2,0.2) --(8.2,-0.2);
            \draw [decorate, line width = 1pt,decoration = {brace,mirror}] (-0.5,-0.2) --  (2.3,-0.2);
            \draw [decorate, line width = 1pt,decoration = {brace,mirror}] (3,-0.2) --  (6.3,-0.2);
            \draw [decorate, line width = 1pt,decoration = {brace,mirror}] (7,-0.2) --  (7.6,-0.2);
            \draw [decorate, line width = 1pt,decoration = {brace,mirror}] (8,-0.2) --  (8.2,-0.2);
            \draw (8.4,-0.5) node {$R_4'=\emptyset$};
            \draw [decorate, line width = 1pt,decoration = {brace}] (0,0.2) --  (3,0.2);
            \draw [decorate, line width = 1pt,decoration = {brace}] (3.2,0.2) --  (7,0.2);
            \draw [decorate, line width = 1pt,decoration = {brace}] (7.1,0.2) --  (8,0.2);
            \draw [decorate, line width = 1pt,decoration = {brace}] (8,0.2) --  (8.2,0.2);
            \draw (8.5,0) node {$\ldots$};
        \end{tikzpicture}
    \end{center}
    Denote
    \begin{itemize}
        \item the set of elements $i$ in $\{1,\ldots,n\}$ such that $w(m+i)\prec w(\sigma-n)$ by $R_1'$,
        \item the set of elements $i$ in $\{1,\ldots,n\}$ such that $w(m + |R_1'|) \prec w(m+i)\prec w(\sigma-n+|R_1'|)$ by $R_2'$,
        \item the set of elements $i$ in $\{1,\ldots,n\}$ such that $w(m + |R_1'|+|R_2'|) \prec w(m+i)\prec w(\sigma-n+|R_1'|+|R_2'|)$ by $R_3'$,
        \item $\ldots$
        \item the set of elements $i$ in $\{1,\ldots,n\}$ such that $w(\sum_{j=1}^{k-1}|R_j'|) \prec w(m+i)\prec w(\sigma-n+\sum_{j=1}^{k-1}|R_j'|)$ by $R_{k}'$,
        \item The head ends where $R'_{k+1} = \emptyset$. 
    \end{itemize}
    Then the set of all dead ``dark gray'' elements in the head is the union of all $R'_j$'s. 
\end{proposition}
\begin{proof}
    In Case III, since $\sigma>n$, all the elements between 1 and $\sigma-n$ are dead, as well as the elements from $\{m+1,\ldots,m+n\}$ interlacing the elements $1,\ldots,\sigma-n$. Therefore, we can imitate the process in Proposition \ref{prophead} to exhaust the dead elements in the head for Case III. The process finishes when one of the following three situations occur:
    \begin{enumerate}
        \item All ``light gray'' elements from $\{1,\ldots,m\}$ are finished.
        \item All ``dark gray'' elements from $\{m+1,\ldots,n\}$ are finished.
        \item $R_{k+1}' = \emptyset$ for some $k$.
    \end{enumerate}
    Similar to Proposition \ref{prophead}, if (3) occurs before (1) or (2), we can construct a permissible move involving  the element $\sum_{j=1}^k |R_j'|+1$.
\end{proof}
The lengths of the $R_j'$ intervals satisfy the following lemma:
\begin{lemma}\label{lemmahead2}
    If for any $1\leq j\leq k$, $R'_i\neq \emptyset$, we have
    \begin{itemize}
        \item $i \leq w(i) \leq i + |R_1'|$ for $i\in \{1,\ldots,\sigma-n\}$
        \item $i + |R_1'|\leq w(i) \leq i + |R_1'|+|R_2'|$ for $i\in \sigma - n + R_1'$,
        \item $i + |R_1'|+|R_2'| \leq w(i) \leq i + |R_1'|+|R_2'|+|R_3'|$ for $i\in  \sigma - n + R_2'$,
        \item \ldots
        \item $i + \sum_{j=1}^{k-1}|R_j'| \leq w(i) \leq i + \sum_{j=1}^{k}|R_j'|$ for $i\in \sigma-n+R_{k-1}'$,
        \item $w(i) = i + \sum_{j=1}^{k}|R_j'|$ for $i\in \sigma-n+R_{k}'$,
    \end{itemize}
    At the right endpoint of each interval, the left part of the inequalities are strict.
\end{lemma}
\begin{proof}
    The lemma follows easily from counting. The strict inequality at the right endpoint of each interval follows from the nonzeroness of $|R_j'|$ for $j\leq k$.
\end{proof}

\subsubsection{Tail Intervals for Case I, III}
We can construct the tail of the dead elements symmetric to the situation described in Lemma \ref{prophead3}:
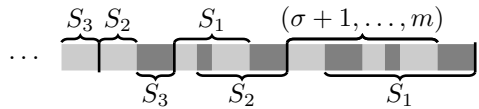
\begin{center}
    \begin{tikzpicture}[line width=10pt]
        \draw[color = black, line width=1pt] (0.5,0.2) --(0.5,-0.2);
        \draw[opacity = 0.5] (0,0)--(0.5,0);
        \draw[opacity = 0.2] (-0.5,0)--(0,0);
        \draw[opacity = 0.5] (-0.7,0)--(-0.5,0);
        \draw[opacity = 0.2] (-1,0)--(-0.7,0);
        \draw[opacity = 0.5] (-1.5,0)--(-1,0);
        \draw[opacity = 0.2] (-2,0)--(-1.5,0);
        \draw[color = black, line width=1pt] (-2,0.2) --(-2,-0.2);
        \draw[opacity = 0.5] (-2.5,0)--(-2,0);
        \draw[opacity = 0.2] (-3,0)--(-2.5,0);
        \draw[opacity = 0.5] (-3.2,0)--(-3,0);
        \draw[opacity = 0.2] (-3.5,0)--(-3.2,0);
        \draw[color = black, line width=1pt] (-3.5,0.2) --(-3.5,-0.2);
        \draw[opacity = 0.5] (-4,0)--(-3.5,0);
        \draw[opacity = 0.2] (-4.5,0)--(-4,0);
        \draw[color = black, line width=1pt] (-4.5,0.2) --(-4.5,-0.2);
        \draw[opacity = 0.2] (-5,0)--(-4.5,0);
        \draw [decorate, line width = 1pt,decoration = {brace}] (-2,0.2) --  (0,0.2);
        \draw (-1,0.5) node {$(\sigma+1,\ldots,m)$};
        \draw (-0.5,-0.5) node {$S_1$};
        \draw [decorate, line width = 1pt,decoration = {brace,mirror}] (-1.5,-0.2) --  (0.5,-0.2);
        \draw [decorate, line width = 1pt,decoration = {brace}] (-2,0.2) --  (0,0.2);
        \draw (-3,0.5) node {$S_1$};
        \draw (-2.6,-0.5) node {$S_2$};
        \draw [decorate, line width = 1pt,decoration = {brace}] (-3.5,0.2) --  (-2.5,0.2);
        \draw [decorate, line width = 1pt,decoration = {brace,mirror}] (-3.2,-0.2) --  (-2,-0.2);
        \draw (-3.75,-0.5) node {$S_3$};
        \draw (-4.25,0.5) node {$S_2$};
        \draw [decorate, line width = 1pt,decoration = {brace}] (-4.5,0.2) --  (-4,0.2);
        \draw [decorate, line width = 1pt,decoration = {brace,mirror}] (-4,-0.2) --  (-3.5,-0.2);
        \draw (-4.75,0.5) node {$S_3$};
        \draw [decorate, line width = 1pt,decoration = {brace}] (-5,0.2) --  (-4.5,0.2);
        \draw (-5.5,0) node {$\ldots$};
    \end{tikzpicture}
\end{center}
\begin{proposition}
    The \emph{tail} can be constructed through the following algorithm:
    Denote 
    \begin{itemize}
        \item the elements $j$ in $\{m+1,\ldots,m+n\}$ such that $w(\sigma+1)\prec w(j)$ by $S_1$,
        \item the elements $j$ in $\{m+1,\ldots,m+n\}$ such that $w(\sigma-|S_1|+1)\prec w(j)\prec w(\sigma)$ by $S_2$,
        \item the elements $j$ in $\{m+1,\ldots,m+n\}$ such that $w(\sigma-|S_1|-|S_2|+1)\prec w(j)\prec w(\sigma-|S_1|)$ by $S_3$,
        \item $\ldots$
        \item the elements $j$ in $\{m+1,\ldots,m+n\}$ such that $w(\sigma-\sum_{j=1}^{k-1}|S_j|+1)\prec w(j)\prec w(\sigma-\sum_{j=1}^{k-2}|S_j|)$ by $S_k$
    \end{itemize}
    There exists a minimal $k$ such that $S_{k+1} = \emptyset$, and the set of all dead light gray elements in the tail is the union of all nonempty $S_j$.
\end{proposition}
\begin{proof}
    In this situation, all the elements between $\sigma+1$ and $m$ are dead. Therefore, the dead elements coming from $\{m+1,\ldots,m+n\}$ are those which interlace with the elements $\sigma+1,\ldots,m$. Following the same procedure as in the proof of Proposition \ref{prophead3} from right to left along the tail, we can follow the process described in the statement of the proposition until one of the following three situations occur
    \begin{enumerate}
        \item All ``light gray'' elements from $\{1,\ldots,m\}$ are finished.
        \item All ``dark gray'' elements from $\{m+1,\ldots,m+n\}$ are finished.
        \item $S_{k+1} = \emptyset$ for some $k$.
    \end{enumerate}
    If (3) happens before (1) and (2), the adjacent first element to the left of $S_k$ is the last living element.
\end{proof}
Denoting the index of the right-most living element from $(m+1)\ldots (m+n)$ by $q$, we have the following lemma
\begin{lemma}\label{lemmatail1}
    For any $1\leq j\leq k$ such that $S_j\neq \emptyset$, we have
    \begin{itemize}
        \item $w(i) =  q-m + i$ for $i\in S_k$,
        \item $w(i) \geq q-m +i$ for $i\in S_{k-1}$,
        \item $w(i) \geq  q-m + |S_{k}| +i$ for $i\in S_{k-2}$,
        \item \ldots
        \item $w(i) \geq q-m+ \sum_{j=3}^{k}|S_j|+i$ for $i\in S_1$.
    \end{itemize}
\end{lemma}
\begin{proof}
    The proof follows from the similar enumeration technique as in Lemma \ref{lemmahead2}.
\end{proof}

\subsubsection{Tail Intervals for Case II}
We can construct the tail of the dead elements symmetric to the situation in Proposition \ref{prophead}:
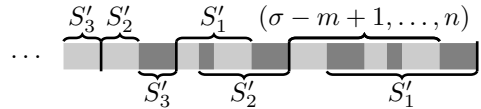
\begin{center}
    \begin{tikzpicture}[line width=10pt]
        \draw[color = black, line width=1pt] (0.5,0.2) --(0.5,-0.2);
        \draw[opacity = 0.5] (0,0)--(0.5,0);
        \draw[opacity = 0.2] (-0.5,0)--(0,0);
        \draw[opacity = 0.5] (-0.7,0)--(-0.5,0);
        \draw[opacity = 0.2] (-1,0)--(-0.7,0);
        \draw[opacity = 0.5] (-1.5,0)--(-1,0);
        \draw[opacity = 0.2] (-2,0)--(-1.5,0);
        \draw[color = black, line width=1pt] (-2,0.2) --(-2,-0.2);
        \draw[opacity = 0.5] (-2.5,0)--(-2,0);
        \draw[opacity = 0.2] (-3,0)--(-2.5,0);
        \draw[opacity = 0.5] (-3.2,0)--(-3,0);
        \draw[opacity = 0.2] (-3.5,0)--(-3.2,0);
        \draw[color = black, line width=1pt] (-3.5,0.2) --(-3.5,-0.2);
        \draw[opacity = 0.5] (-4,0)--(-3.5,0);
        \draw[opacity = 0.2] (-4.5,0)--(-4,0);
        \draw[color = black, line width=1pt] (-4.5,0.2) --(-4.5,-0.2);
        \draw[opacity = 0.2] (-5,0)--(-4.5,0);
        \draw [decorate, line width = 1pt,decoration = {brace}] (-2,0.2) --  (0,0.2);
        \draw (-1,0.5) node {$(\sigma-m+1,\ldots,n)$};
        \draw (-0.5,-0.5) node {$S_1'$};
        \draw [decorate, line width = 1pt,decoration = {brace,mirror}] (-1.5,-0.2) --  (0.5,-0.2);
        \draw [decorate, line width = 1pt,decoration = {brace}] (-2,0.2) --  (0,0.2);
        \draw (-3,0.5) node {$S_1'$};
        \draw (-2.6,-0.5) node {$S_2'$};
        \draw [decorate, line width = 1pt,decoration = {brace}] (-3.5,0.2) --  (-2.5,0.2);
        \draw [decorate, line width = 1pt,decoration = {brace,mirror}] (-3.2,-0.2) --  (-2,-0.2);
        \draw (-3.75,-0.5) node {$S_3'$};
        \draw (-4.25,0.5) node {$S_2'$};
        \draw [decorate, line width = 1pt,decoration = {brace}] (-4.5,0.2) --  (-4,0.2);
        \draw [decorate, line width = 1pt,decoration = {brace,mirror}] (-4,-0.2) --  (-3.5,-0.2);
        \draw (-4.75,0.5) node {$S_3'$};
        \draw [decorate, line width = 1pt,decoration = {brace}] (-5,0.2) --  (-4.5,0.2);
        \draw (-5.5,0) node {$\ldots$};
    \end{tikzpicture}
\end{center}
\begin{proposition}
    The \emph{tail} can be constructed through the following algorithm:
    Denote 
    \begin{itemize}
        \item the elements $j$ in $\{1,\ldots,m\}$ such that $w(\sigma-m +1)\prec w(j)$ by $S_1'$,
        \item the elements $j$ in $\{1,\ldots,m\}$ such that $w(\sigma-m - |S_1'|+1)\prec w(j)\prec w(\sigma-m)$ by $S_2'$,
        \item the elements $j$ in $\{1,\ldots,m\}$ such that $w(\sigma-m-|S_1'|-|S_2'|+1)\prec w(j)\prec w(\sigma-m-|S_1'|)$ by $S_3'$,
        \item $\ldots$
        \item the elements $j$ in $\{1,\ldots,m\}$ such that $w(\sigma-m-\sum_{j=1}^{k-1}|S_j'|+1)\prec w(j)\prec w(\sigma-m-\sum_{j=1}^{k-2}|S_j'|)$ by $S_k'$
    \end{itemize}
    There exists a minimal $k$ such that $S_{k+1}' = \emptyset$, and the set of all dead light gray elements in the tail is the union of all nonempty $S_j'$.
\end{proposition}
\begin{proof}
    The proof follows the same procedure as in the proof of Proposition \ref{prophead} from right to left along the tail. Similarly, the algorithm terminates when one of the following three situations occur:
    \begin{enumerate}
        \item All ``light gray'' elements from $\{1,\ldots,m\}$ are finished.
        \item All ``dark gray'' elements from $\{m+1,\ldots,m+n\}$ are finished.
        \item $S'_{k+1} = \emptyset$ for some $k$.
    \end{enumerate}
    The adjacent element to the left of $S_k'$ does not interlace with any dead interval, and thus is an living element.
\end{proof}
Denoting the index of the right-most living element from $1\ldots m$ by $q'$, we have the following lemma
\begin{lemma}\label{lemmatail2}
    For any $1\leq i\leq n$ such that $S_j'\neq \emptyset$, we have
    \begin{itemize}
        \item $n-\sigma+q'+ |S_n'| + i < w(i) \leq n-\sigma+q'+ |S_n'| + |S_{n-1}'| + i$ for $i\in S'_n$,
        \item $ n-\sigma+q'+ |S_n'| + |S_{n-1}'| + i < w(i) \leq n-\sigma+q'+ \sum_{j={n-2}}^n |S_j'| + i$ for $i\in S'_{n-1}$,
        \item \ldots
        \item $w(i) >  n-\sigma+q' + \sum_{j=1}^{n}|S_j'| + i$ for $i\in S_1'
        $.
    \end{itemize}
\end{lemma}
\begin{proof}
    The proof follows from the similar enumeration technique as in Lemma \ref{lemmahead}.
\end{proof}

\section{\texorpdfstring{$L$}{L}-function Calculations}
\subsection{\texorpdfstring{$L$}{L}-functions}
For self-dual cuspidal automorphic representations $\tau$, the Rankin-Selberg $L$-functions $L(s,\tau\times\hat{\tau})$ satisfies the following functional equation
\[
    L(s,\tau\times\hat{\tau}) = \epsilon(s,\tau\times\hat{\tau})L(1-s,\tau\times\hat{\tau})
\]
with the epsilon factor $\epsilon = c^{s-\frac{1}{2}}$ for some rational number $c$. Without incurring any confusion, we will simply denote $L(s,\tau\times\hat{\tau})$ and $\epsilon(s,\tau\times\hat{\tau})$ by $L(s)$ and $\epsilon(s)$, respectively. By \cite{moeglin1989spectre}, the normalizing factor defined in the previous section is equal to the product
\begin{align*}
    r(w,\underline{s}) &= \prod_{\substack{i<j\\w(i)>w(j)}} \frac{L(\nu_i-\nu_j)}{L(1+\nu_i-\nu_j)\epsilon(\nu_i-\nu_j,\tau\times\hat{\tau})}.
\end{align*}
In the situation we are interested in, we specialize the parameters to the following values:
\begin{align*}
    \nu_i &= \frac{1-m}{2}+i-1+s_1, \\
    \nu_j &= \frac{1-n}{2}+j-m-1+s_2,\\
    s &= s_1-s_2.
\end{align*}
For any element $w$ representable by the diagram
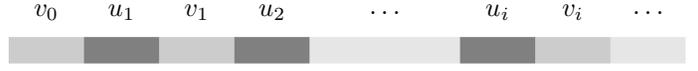
\begin{center}
    \begin{tikzpicture}[line width=10pt]
        \draw (0.5,0.5) node {$v_0$};
        \draw[opacity = 0.2] (0,0) --(1,0);
        \draw (1.5,0.5) node {$u_1$};
        \draw[opacity = 0.5] (1,0) --(2,0);
        \draw (2.5,0.5) node {$v_1$};
        \draw[opacity = 0.2] (2,0) --(3,0);
        \draw (3.5,0.5) node {$u_2$};
        \draw[opacity = 0.5] (3,0) --(4,0);
        \draw (5,0.5) node {$\ldots$};
        \draw[opacity = 0.1] (4,0) --(6,0);
        \draw (6.5,0.5) node {$u_i$};
        \draw[opacity = 0.5] (6,0) --(7,0);
        \draw (7.5,0.5) node {$v_i$};
        \draw[opacity = 0.2] (7,0) --(8,0);
        \draw (8.5,0.5) node {$\ldots$};
        \draw[opacity = 0.1] (8,0) --(9,0);
    \end{tikzpicture},
\end{center}
the collection of inverted pairs $\mathrm{inv}(w)=\{(i,j)\mid i<j, w(i)>w(j)\}$ is the set
\[
   \mathrm{inv}(w) =  (v_1\times u_1)\bigcup (v_2\times(u_1\cup u_2) )\bigcup (v_3\times(u_1\cup u_2\cup u_3))\bigcup\ldots.
\]
Since $\nu_i-\nu_j = s+\frac{m+n}{2} + i-j$, 
For any $w\in \mathbb{S}_m\times\mathbb{S}_n\backslash\mathbb{S}_{m+n}$, we denote by $i_w$ the smallest $i$ such that there exists a $j>i$ with $w(j)<w(i)$. For each fixed $i\geq i_w$, denoting the largest $j$ such that $(i,j)\in \mathrm{inv}(w)$ by $j_i$, the product corresponding to all $j$'s such that $(i,j)\in \mathrm{inv}(w)$ is
\begin{align*}
     &\prod_{w(j)<w(i)}  \frac{L(\nu_i-\nu_j)}{L(1+\nu_i-\nu_j)\epsilon(\nu_i-\nu_j)}\\
    =& \frac{L(s+\frac{m+n}{2} + i-(m+1))}{L(s+\frac{m+n}{2} + i-(m+1)+1)} \times\ldots\times\frac{L(s+\frac{m+n}{2} + i-j_i)}{L(s+\frac{m+n}{2} + i-j_i+1)}\times\frac{1}{\text{products of }\epsilon\text{'s}}\\
    =& \frac{L(s+\frac{m+n}{2} + i-j_i)}{L(s-\frac{m-n}{2}+i)} \times\frac{1}{\text{products of }\epsilon\text{'s}}.
\end{align*}
Replacing $s$ by $s = \frac{m+n}{2} -\sigma + t$ and taking into account that
\[
    w(i) = i + j_i - m,
\]
after plugging in all the parameters, the product of $\epsilon$-factors can be expressed as
\[
    \rho_i(w,t)=c^{\sum_{k=m+1}^{j_i}(s+\frac{m+n-1}{2}+i-k)} = c^{\left(w(i)-i\right) (n-\sigma+i-1+t)-\frac{1}{2} \left(w(i)-i\right){}^2}, 
\]
and the normalization factor of the intertwining operator is thus equal to
\begin{equation}\label{assemble}
    r(w,t) = \prod_{i=i_w}^m \frac{L(n-\sigma+2i-w(i)+t)}{L(n-\sigma+i+t)}\rho_i(w,t)^{-1}.
\end{equation}
For any $\sigma$, we will be only interested in the analytic property of $r(w,t)$ in the region $t\geq 0$, in which there will be no cancellation of the poles of the denominator by the zeros of the denominator in the critical strip.\\

The following lemma is a generalization of \cite[Lemma 8.5]{hanzer2015images}:

\begin{lemma}\label{lemc}
For any section $f\in I(\tau,\underline{s})$, the images of the normalized intertwining operator $N(w,\underline{s})f$ are equal as long as $w$ lies in the same orbit $\mathcal{O}_{\underline{s}}$.
\end{lemma}
\begin{proof}
    We would like to compare two $N(w_1,\underline{s})$ and $N(w_2,\underline{s})$ for different $w_1,w_2$ with isomorphic twists \[\rho_{m,n}(\tau,\underline{s})^{w_1} \cong \rho_{m,n}(\tau,\underline{s})^{w_2}.\] The normalized intertwining operators are holomorphic in the dominant chamber, and each normalized intertwining operator
\[
    N(w_i,\underline{s}): \Ind_P^G \rho_{m,n}(\tau,\underline{s}) \longrightarrow \Ind_P^G \rho_{m,n}(\tau,\underline{s})^{w_i},
\]
is invertible. Since the cuspidal representations $\tau$ are unitary, and by \cite[Proposition 38]{knapp1971interwining} the normalized intertwining operators satisfy
\begin{enumerate}
    \item $N(w_1w_2,\underline{s}) = N(w_1,w_2\underline{s})N(w_2,\underline{s})$
    \item $N(w,\underline{s})^* = N(w^{-1},-\overline{\underline{s}})$,
    \item $N(w,\underline{s})$ is unitary for $s = -\overline{s}$.
\end{enumerate}
Since we can factorize the elements $w_2 =  u w_1$ such that $u$ is a permissible move, which is known to be an involution. The intertwining operator
\[
    N(u,w_2\underline{s}): \Ind_P^G \rho_{m,n}(\tau,\underline{s})^{w_2}\longrightarrow \Ind_P^G \rho_{m,n}(\tau,\underline{s})^{w_1}
\]
is thus a self-adjoint, unitary involution. The operator $N(u,w_2\underline{s})$ can be factorized further into:
\[
    N(u,w_2\underline{s}) = N(v^{-1},vw_2\underline{s}) N(\iota, v w_2\underline{s})N(v,w_2\underline{s})
\]
where $v$ is the Weyl group element sending each pairs of corresponding blocks in the permissible move to adjacent blocks, and the operator $N(\iota, v w_2\underline{s})$ is the transposition of these adjacent blocks. In fact, the intertwining operator $N(\iota, v w_2\underline{s})$ is a product of self-intertwining operators $N(0)$ of $\Ind_{P_{[a,a]}}^{G_{2a}}\left(\tau\boxtimes \tau\right)$ corresponding to the Weyl group element swapping the two Levi blocks. By \cite[Proposition 6.3]{keys1988artin}, such an operator $N(0) = 1$, and thus $N(u,w_2\underline{s}) = 1$.
\end{proof}
As a result of this lemma, we can collect the normalization factors in the constant term of $E^{m,n}(\otimes_v f_v,\underline{s})$ into sums over orbits, as in
\begin{equation}\label{constmn}
    c_{U_{m+n}}E^{m,n}(\otimes_v f_v,\underline{s}) = \sum_{\mathcal{O}_{\underline{s}}(w)}\left(\sum_{w'\in \mathcal{O}_{\underline{s}}(w)} r(w',\underline{s})\right) \bigotimes_{v}N_v(w',\underline{s})f_v.
\end{equation}
Each sum over an orbit is denoted by $R(w,\underline{s})=\sum_{w'\in \mathcal{O}_{\underline{s}}(w)} r(w',\underline{s})$. We will discuss the cancellations of poles in the sum $R(w,\underline{s})$.

\subsection{Calculation on a Single Living Interval}
Now we calculate the sums on a single living interval. In this section, following Proposition \ref{segmentation}, for each living interval $K_i$ with a starting interval $\Xi_{k_{i-1}+1}$ and ending interval $\Xi_{k_i}$, we denote these individual green intervals by
\[
    \Xi_{k_{i-1}+1}, \Xi_{k_{i-1}+2}, \ldots, \Xi_{k_{i}}
\]
with $|\Xi_{k_{i-1}+1}|+\ldots+|\Xi_{k_{i}}|= |K_i|$. Their corresponding red intervals are denoted by
\[
    \Sigma_{k_{i-1}+1}, \Sigma_{k_{i-1}+2}, \ldots, \Sigma_{k_{i}}.
\]
\\
\indent Denoting by $A_p(t)$ the product of $L(n-\sigma+2i-w(i)+t)\rho_i(w,t)^{-1}$ along each $K_p$, and $B_p(t)$ the product of $L(n-\sigma+2i-w(m+n-\sigma+i)+t)\rho_i(w,t)^{-1}$ along each $L_p$. Since $i = s(\Xi_{k_{p-1}+1}) + r$ and $w(i) = n - \sigma + 2s(\Xi_{k_{p-1}+1}) + \sum_{\Sigma_{k_{p-1}+s}\prec w(i)} |\Sigma_{k_{p-1}+s}| + r$, then
\[
    n-\sigma+2i -w(i) = r - \sum_{\Sigma_{k_{p-1}+s}\prec w(i)} |\Sigma_{k_{p-1}+s}|.
\]
After swapping the corresponding green and red intervals, we obtain
\[
    w(m+n-\sigma+i) = n - \sigma + 2s(\Xi_{k_{p-1}+1}) + \sum_{\Xi_{k_{p-1}+s}\prec w(m+n-\sigma+i)} |\Xi_{k_{p-1}+s}| + r
\]
and thus
\[
    n-\sigma+2i -w(m+n-\sigma+i) = r - \sum_{\Xi_{k_{p-1}+s}\prec w(m+n-\sigma+i)} |\Xi_{k_{p-1}+s}|.
\]
\subsection{Product of  \texorpdfstring{$L$}{L} Factors}
In this section we will calculate $A_p(t)$ and $B_p(t)$. We understand that some of the $\Xi_i, \Sigma_i$ intervals group together on the same $v,u$ interval. If we denote the endpoints of these intervals by
\[
    s(\Xi_{k_{p-1}+1}) + \{a_1,\ldots,a_r\},
\]
then the corresponding product of $L$-functions in $A_p(t)$ and $B_p(t)$ are given by
\[
    A_p'(t) = L(1+t)\ldots L(a_1+t)L(a_1-b_1+1+t)\ldots L(a_2-b_1+t)\ldots L(a_{r-1}-b_{r-1}+1+t)\ldots L(a_{r}-b_{r-1}+t)
\]
and
\[
    B_p'(t) = L(-a_1+1+t)\ldots L(b_1-a_1+t)L(b_1-a_2+1+t)\ldots L(b_2-a_2+t)\ldots L(b_{r-1}-a_{r}+1+t)\ldots L(b_{r}-a_{r}+t).
\]
We can compare these two products, and it turns out that:
\begin{lemma}
    The two products above satisfy \[A'_p(t) = c^{\sum_r\left(\frac{\sum_{\Xi_{k_{p-1}+s}\prec w(m+n-\sigma+i)}|\Xi_{k_{p-1}+s}|-\sum_{\Sigma_{k_{p-1}+s}\prec w(i)}|\Sigma_{k_{p-1}+s}|}{2}\right)}B'_p(-t).\]
\end{lemma}
\begin{proof}
    Since each $b_i\prec a_i$, we can reorder the factors in $A'_p(t)$ and $B'_p(-t)$. If we represent these factors in $A'_p(t)$ as intervals, and they can be illustrated in the following diagram:
    \begin{center}
        \begin{tikzpicture}[line width=5pt]
            \draw[opacity = 0.6] (0,0) --(2,0);
            \draw (-3,0) node {$L(1+t)\ldots L(a_1+t)$};
            \draw[opacity = 0.6] (1,-0.5) --(3,-0.5);
            \draw (-2,-0.5) node {$L(a_1-b_1+1+t)\ldots L(a_2-b_1+t)$};
            \draw[opacity = 0.6] (2,-1) --(4,-1);
            \draw (-1,-1) node {$\ldots$};
            \draw (4,-1.5) node {$\ldots$};
            \draw[opacity = 0.6] (5,-2) --(7,-2);
            \draw (1,-2) node {$L(a_{r-1}-b_{r-1}+1+t)\ldots L(a_{r}-b_{r-1}+t)$};
        \end{tikzpicture},
    \end{center}
    while the factors in $B'_p(t)$ correspond to the intervals
    \begin{center}
        \begin{tikzpicture}[line width=5pt]
            \draw[opacity = 0.6] (1,0) --(2,0);
            \draw (-3,0) node {$L(-a_1+1+t)\ldots L(b_1-a_1+t)$};
            \draw[opacity = 0.6] (2,-0.5) --(3,-0.5);
            \draw (-2,-0.5) node {$L(b_1-a_2+1+t)\ldots L(b_2-a_2+t)$};
            \draw[opacity = 0.6] (3,-1) --(4,-1);
            \draw (-1,-1) node {$\ldots$};
            \draw (4,-1.5) node {$\ldots$};
            \draw[opacity = 0.6] (0,-2) --(7,-2);
            \draw (-3,-2) node {$L(b_{r-1}-a_{r}+1+t)\ldots L(b_{r}-a_{r}+t)$};
        \end{tikzpicture}
    \end{center}
between which a bijection can be established after sending each index $p$ to $1-p$. By the functional equation of the Rankin-Selberg $L$-function $L(s,\tau\times\hat\tau)$, we have
\[
    A'_p(t) = c^{\sum_r\left(r - \sum_{\Sigma_{k_{p-1}+s}\prec w(i)}|\Sigma_{k_{p-1}+s}|-1/2 + t\right)}B'_p(-t).
\]
Applying the functional equation again, we can see that
\[
    1 = c^{\sum_r\left(r - \sum_{\Sigma_{k_{p-1}+s}\leq w(i)}|\Sigma_{k_{p-1}+s}|-1/2 + t\right)}c^{\sum_r\left(r - \sum_{\Xi_{k_{p-1}+s}\leq w(m+n-\sigma+i)}|\Xi_{k_{p-1}+s}|-1/2 - t\right)}.
\]
Therefore,
\[
    \sum_r \left(r-\frac{\sum_{\Sigma_{k_{p-1}+s}\prec w(i)}|\Sigma_{k_{p-1}+s}|+\sum_{\Xi_{k_{p-1}+s}\prec w(m+n-\sigma+i)}|\Xi_{k_{p-1}+s}|}{2}-\frac{1}{2}\right) = 0
\]
and the original factor becomes
\[
    A'_p(t) = c^{\sum_r\left(\frac{\sum_{\Xi_{k_{p-1}+s}\prec w(m+n-\sigma+i)}|\Xi_{k_{p-1}+s}|-\sum_{\Sigma_{k_{p-1}+s}\prec w(i)}|\Sigma_{k_{p-1}+s}|}{2}\right)}B'_p(-t).
\]
\end{proof}

\subsection{Product of \texorpdfstring{$\rho$}{rho}-factors}
This subsection calculates the cancellations of poles in $R(w,\underline{s})$ from the living intervals.
\begin{proposition}
    Along each living interval, $A_p(t) + B_p(t)$ is holomorphic.
\end{proposition}
\begin{proof}
Each individual $\rho$-factor can be calculated with the following procedure:\\

For any $i$, we have
\begin{align*}
    \rho_i(w,t) &= c^{(w(i)-i)(n-\sigma+i-1+t) - \frac{1}{2}(w(i)-i)^2}\\
    &=c^{(n-\sigma+s(\Sigma_{k_{p-1}+1})+\sum_i|\Sigma_i|)(n-\sigma+i-1+t) - \frac{1}{2}(n-\sigma+s(\Sigma_{k_{p-1}+1})+\sum_i|\Sigma_i|)^2}
\end{align*}
and
\begin{align*}
    \rho'_i(w,t) &= c^{(w(m+n-\sigma+i)-i)(n-\sigma+i-1+t) - \frac{1}{2}(w(m+n-\sigma+i)-i)^2}\\
    &=c^{(n-\sigma+s(\Sigma_{k_{p-1}+1})+\sum_i|\Xi_i|)(n-\sigma+i-1+t) - \frac{1}{2}(n-\sigma+s(\Sigma_{k_{p-1}+1})+\sum_i|\Xi_i|)^2}.
\end{align*}
The differences of their products is
\begin{align*}
    c^{\sum_i(\sum_j(|\Sigma_j|-|\Xi_j|))(n-\sigma+s+r-1+t)-(\sum_j(|\Sigma_j|-|\Xi_j|)))(n-\sigma+s+\frac{|\Sigma_i|+|\Xi_i|}{2})} = c^{\sum_i(\sum_j(|\Sigma_j|-|\Xi_j|))(r-1+t-\frac{|\Sigma_i|+|\Xi_i|}{2})}
\end{align*}
Therefore, 
\[
    A_p(t) = A'_p(t)\rho_i(w,t)^{-1} = A'_p(t)c^{-(n-\sigma+s(\Sigma_{k_{p-1}+1})+\sum_i|\Sigma_i|)(n-\sigma+i-1+t) + \frac{1}{2}(n-\sigma+s(\Sigma_{k_{p-1}+1})+\sum_i|\Sigma_i|)^2}
\]
and
\[
    B_p(t) = B'_p(t)\rho'_i(w,t)^{-1} = B'_p(t)c^{-(n-\sigma+s(\Sigma_{k_{p-1}+1})+\sum_i|\Xi_i|)(n-\sigma+i-1+t) + \frac{1}{2}(n-\sigma+s(\Sigma_{k_{p-1}+1})+\sum_i|\Xi_i|)^2}.
\]
The sum can be reduced to
\begin{align*}
    &A_p(t) + B_p(t) \\
    &= \left(\prod_i\rho_i'(w,t)^{-1}\right)\left(B_p'(-t)c^{\sum_r\left(\frac{\sum_j\Xi_j-\sum_j\Sigma_j}{2}\right)}c^{\sum_r\left(\sum_j(\Xi_j-\Sigma_j)(r-1+t-\frac{\Sigma_j+\Xi_j}{2})\right)} + B_p'(t)\right)\\
    &=\left(\prod_i\rho_i'(w,t)^{-1}\right)\left(B_p'(-t)c^{\sum_r\left(\sum_j(\Xi_j-\Sigma_j)(r-\frac{1}{2}+t-\frac{\Sigma_j+\Xi_j}{2})\right)} + B_p'(t)\right).
\end{align*}
The exponent summing over all $\sum_j(\Xi_j-\Sigma_j)\left(r-\frac{1}{2}+t-\frac{\Sigma_j+\Xi_j}{2}\right)$ is in fact equal to
\begin{align*}
    \sum_j(\Xi_j-\Sigma_j)\left(r-\frac{1}{2}+t-\frac{\Sigma_j+\Xi_j}{2}\right) &= \sum_{i\in \text{living elements}}(w(i)-w(m+n-\sigma+i))\\
    &\left(n-\sigma+2i-\frac{w(i)+w(m+n-\sigma+i)+1}{2}+t\right).
\end{align*}
Noting that the individual terms in the summation above can be rearranged into the form:
\begin{align*}
    &-\sum_i\frac{\left(n-\sigma+2i-w(i)\right)^2 - \left(n-\sigma+2i-w(m+n-\sigma+i)\right)^2}{2}\\
     + &\sum_i\frac{\left(n-\sigma+2i-w(i)\right) - \left(n-\sigma+2i-w(m+n-\sigma+i)\right)}{2}\\
     - t&\sum_i\left(\left(n-\sigma+2i-w(i)\right) - \left(n-\sigma+2i-w(m+n-\sigma+i)\right)\right).
\end{align*}
The first sum can be further rearranged into the form of $k^2 - (1-k)^2 = 2k-1$, which cancels out the second sum. Therefore,
\[
    A_p(t) + B_p(t) = \left(\prod_i\rho_i'(w,t)^{-1}\right)\left(c^{-t \alpha}B_p'(-t) + B_p'(t)\right).
\]
Since both $A_p'(t)$ and $B_p'(t)$ has a simple pole at $t=0$, and since $c^{-t\alpha}$ is holomorphic with zeroth order coefficient 1, we conclude that $A_p(t)+B_p(t)$ is holomorphic.
\end{proof}

\subsection{Other Factors and Proof of the Theorem}
Now we finish the calculation of the sum
\[
    R(w,\underline{s}) = \sum_{w'\in\mathcal{O}_s(w)} r(w',\underline{s})
\]
outside of the contributions of the living intervals. Recall that $i_w$ is the first index $i$ such that there exist a $j>i$ such that $w(i)>w(j)$. The location of the first living interval may affect $i_w$ when $1$ is a living element. When that happens, we are forced to require $m+n-\sigma+1 = m+1$, in which case $n = \sigma$. 
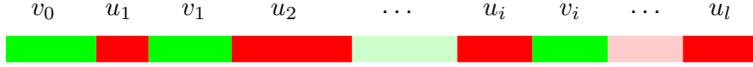
\begin{center}
        \begin{tikzpicture}[line width=10pt]
            \draw (0.5,0.5) node {$v_0$};
            \draw[color = green] (0,0) --(1.2,0);
            \draw (1.5,0.5) node {$u_1$};
            \draw[color = red] (1.2,0) --(1.9,0);
            \draw (2.5,0.5) node {$v_{1}$};
            \draw[color = green] (1.9,0) --(3,0);
            \draw (3.7,0.5) node {$u_{2}$};
            \draw[color = red] (3,0) --(4.6,0);
            \draw[color = green, opacity = 0.2] (4.6,0) --(6,0);
            \draw (5.2,0.5) node {$\ldots$};
            \draw (6.5,0.5) node {$u_i$};
            \draw[color = red] (6,0) --(7,0);
            \draw (7.5,0.5) node {$v_i$};
            \draw[color = green] (7,0) --(8,0);
            \draw (8.5,0.5) node {$\ldots$};
            \draw[color = red, opacity = 0.2] (8,0) --(9,0);
            \draw (9.5,0.5) node {$u_l$};
            \draw[color = red] (9,0) --(10,0);
        \end{tikzpicture}
\end{center}
In this case, if $1$ is a living element for $w$, $R(w,\underline{s})$ can be separated into two terms $R'(t)$ and $R''(t)$ which consist of the products of $L$-function factors with $i_w'$ for $w'\in \mathcal{O}_s(w)$ the same as the $i_{w_0}$ of the basepoint element $w_0$ or not, respectively. 
\[
    R'(t) = \sum_{\substack{w'\in\mathcal{O}_s(w)\\ i_{w'}= i_{w_0}}} r(w',\underline{s}) = \frac{\prod_{i\in \mathrm{head}\cup\mathrm{tail}} L(n-\sigma+2i-w_0(i)+t)\rho_i^{-1}(w_0,t)}{\prod_{i=v_0+1}^{m}L(n-\sigma+i+t)} \prod_{i=1}^r(A_i(t) + B_i(t)) .
\] 
and the sum of the terms $r(w',s)$ satisfying $i_{w'} > 1$ by
\[
    R''(t) = \sum_{\substack{w'\in\mathcal{O}_s(w)\\ i_{w'}\neq i_{w_0}}} r(w',\underline{s}) = \frac{\prod_{i\in \mathrm{head}\cup\mathrm{tail}} L(n-\sigma+2i-w_0(i)+t)\rho_i^{-1}(w_0,t)}{\prod_{i=1}^{m} L(n-\sigma+i+t)} B_0(t)\prod_{i=1}^r(A_i(t) + B_i(t)),
\] 
then
\[
    R(w,\underline{s}) = R'(t) + R''(t).
\]
In this case, there are no indices to the left of all the movable indices. The denominator is the product $L(1+t)\ldots L(m+t)$, and the tail does not contribute to any poles since $w_0(i)<2i$ on tails.\\

On the other hand, if $1$ is not a movable element, then $R(w,\underline{s})$ simply takes the form
\[
    R(w,\underline{s}) = \sum_{\substack{w'\in\mathcal{O}_s(w)}} r(w',\underline{s}) = \frac{\prod_{i\in \mathrm{head}\cup\mathrm{tail}} L(n-\sigma+2i-w(i)+t)\rho_i^{-1}(w,t)}{\prod_{i=1}^{m} L(n-\sigma+i+t)} \prod_{i=1}^r(A_i(t) + B_i(t)).
\]
We will have to calculate the contribution from the head and the tails separately.
\subsubsection{Head Factors}
Now we calculate the contribution to the poles from the head terms in the numerator. In Case I and II, by Lemma \ref{lemmahead}, we have
\begin{itemize}
    \item $n - \sigma + 2i - w(i) > i$ for $i\in R_1$,
    \item $n - \sigma + 2i - w(i) > i - |R_1|$ for $i\in R_2$,
    \item $\ldots$
    \item $n - \sigma + 2i - w(i) > i - \sum_{j=1}^{k-1}|R_j|$ for $i\in R_k$.
\end{itemize}
Therefore, since the head terms can be expressed as
\[
    \prod_{j=1}^k\prod_{i\in R_j} L(n-\sigma+2i-w(i)+t),
\]
the point $t=0$ is not a pole because all the $n-\sigma+2i-w(i) > 1$ for $i\in \cup_{j=1}^k R_j$. For Case III, by Lemma \ref{lemmahead2}, we have
\begin{itemize}
    \item $n-\sigma+2i-w(i) \leq i - (\sigma-n)$ for $i\in \{1,\ldots,\sigma-n\}$, and when $i = \sigma - n$ the inequality is strict,
    \item $n-\sigma+2i-w(i) \leq i - (\sigma-n) - |R_1'|$ for $i\in \sigma-n+R_1'$, and when $i = \sigma-n+r(R_1')$ the inequality is strict,
    \item $\ldots$
    \item $n-\sigma+2i-w(i) = i - (\sigma-n) - \sum_{j=1}^{k}|R_j'|$ for $i\in \sigma-n+R_k'$, and when $i = \sigma-n+r(R_k')$ we have $n-\sigma+2i-w(i) = 0$.
\end{itemize}
Therefore, in Case III when $\sigma > n$, the only possible contribution to the pole on the numerator occurs when $i = \sigma -n+r(R_k')$, and when $\sigma = n$ there will be no head factors. However, this pole will be cancelled by the denominator when $\sigma\geq n$ as in Case III since we allow $i = \sigma-n$ and $i=\sigma-n+1$ in the factors of the denominator: the denominator $\prod_{i=1}^m L(n-\sigma+i+t)$ has a pole of order 2 at $t=0$ when $\sigma > n$, and a pole of order 1 at $t=0$ when $\sigma = n$. In the case when there are head factors, the contribution to the pole $t=0$ on the numerator will be cancelled by the pole on the denominator.

\subsubsection{Tail Factors}
In the Cases I and III, the contribution to the poles from the tail factors can be calculated with Lemma \ref{lemmatail1}. Since by Lemma \ref{lemmatail1}, on each interval $S_k$, we have
\begin{itemize}
    \item $n - \sigma + 2i - w(i) = i - (q - (m+n-\sigma))$ for $i\in S_k$,
    \item $n - \sigma + 2i - w(i) \geq i - (q - (m+n-\sigma)) > 1$ for $i\in S_{k-1}$,
    \item $\ldots$
    \item $n - \sigma + 2i - w(i) > i - (q - (m+n-\sigma)) - \sum_{j=3}^{k}|S_j| > 1$ for $i\in S_1$,
\end{itemize}
where the strict inequality is due to the nonvanishing of $|S_j|$, it turns out that the only possible contribution to the pole at $t=0$ from the tail factors occurs at $i = (q - (m+n-\sigma))+1$, which is the only contribution to the poles of $R(w,\underline{s})$. However, as was discussed in the previous subsection, in the cases when $\sigma > n$ or when $\sigma = n$ and $m > n$, this pole will be cancelled by the denominator. When $m=n=\sigma$, there will be no head or tail factors. On the other hand, in Case II, by Lemma \ref{lemmatail2}, we have
\begin{itemize}
    \item $n-\sigma+2i-w(i) < i - (q' + |S_n'|) < 0$ for $i\in S_n'$,
    \item $n-\sigma+2i-w(i) < i - (q' + |S_n'|
    +|S_{n-1}'|) < 0$ for $i\in S_{n-1}'$,
    \item $\ldots$
    \item $n-\sigma+2i-w(i) < i - (q' + \sum_{j=1}^n|S_j'|
    ) < 0$ for $i\in S_{1}'$.
\end{itemize}
Therefore, there will be no contribution to the pole from the numerator of $R(w,\underline{s})$. Thus, the tail factors will contribute to the poles of $R(w,\underline{s})$ if and only if $0\leq\sigma < \min\{m,n\}$.

Thus, we have concluded the proof of Theorem \ref{thma} that the possible simple poles occur only when $\sigma \in\{0,1,\ldots,\min\{m,n\}-1\}$.

\subsection{The Residue of \texorpdfstring{$E(\cdot,\underline{s})$}{E(.,s)}}
Fixing any $\sigma\in\{0,1,\ldots,\min\{m,n\}-1\}$, now we will prove Corollary \ref{corb} of Theorem \ref{thma} which describes the residue of $E(\cdot,\underline{s})$ at $s = s_1-s_2 = \frac{m+n}{2}-\sigma$.
\begin{proof}
    For the character \[\underline{\Lambda} = (s_1,\ldots,s_m;t_1,\ldots,t_n),\] denote the full principal series induced from $P_{[1,\ldots,1]}$ by
    \[
        I(\underline{\Lambda}) = \mathrm{Ind}_{P_{m+n}}^{G_{m+n}}\left(\tau|\cdot|^{s_1}\boxtimes\ldots\boxtimes\tau|\cdot|^{s_m}\boxtimes \tau|\cdot|^{t_1}\boxtimes\ldots\boxtimes\tau|\cdot|^{t_n}\right).
    \]
    Assuming $\sigma\in\{0,1\ldots,\min\{n_1,n_2\}-1\}$, we define the following four Weyl group elements:
    \begin{align*} 
        w_1: &(1,2,\ldots,m;m+1,m+2,\ldots,m+n)\mapsto (m,m-1,\ldots,1;m+n,m+n-1,\ldots,m+1)\\
        w_\sigma: & (1,2,\ldots,m;m+1,m+2,\ldots,m+n)\mapsto (m+1,\ldots,m+n-\sigma; 1,\ldots,m;m+n-\sigma+1, \ldots,m+n)\\
        w_\sigma': &(1,2,\ldots,m;m+1,m+2,\ldots,m+n) \mapsto (1,2,\ldots,m;m+\sigma+1,\ldots,m+n;m+1,\ldots,m+\sigma)\\
        w_1' : & (1,2,\ldots,m;m+1,m+2,\ldots,m+n)\mapsto (m+n-\sigma,\ldots,1;m+n-\sigma+1,\ldots,m+n).
    \end{align*}
    These four elements satisfy the relation $w_\sigma w_1 = w_1' w_\sigma'$. Consider following four intertwining operators corresponding to these Weyl group elements:
    \begin{align*}
        \xymatrix{
            I(\underline{\Lambda}) \ar[rr]^{N(w_1,\underline{\Lambda})} \ar[d]_{N(w_\sigma',\underline{\Lambda})} & & I(w_1\underline{\Lambda}) \ar[d]^{N(w_\sigma,w_1\underline{\Lambda})}\\
            I(w_\sigma'\underline{\Lambda})\ar[rr]_{N(w_1',w_\sigma'\underline{\Lambda})}
            & & I(w_\sigma w_1\underline{\Lambda}),
        }
    \end{align*}
    The image of the normalized intertwining operator $N(w_\sigma w_1,\underline{s})$ in $I(w_\sigma w_1\underline{\Lambda})$ is the representation \[I_\sigma = \mathrm{Ind}_{P_{[m,n]}}^{G_{m+n}}\left(\Delta(\tau,m+n-\sigma)|\cdot|^{-\frac{m-n}{2}}\boxtimes \Delta(\tau,\sigma)|\cdot|^{\frac{m-n}{2}}\right),\]which is irreducible when $\sigma\in \{0,1,\ldots,\min\{m,n\}-1\}$ by \cite[I.7, I.11]{moeglin1989spectre} for archimedean places and \cite[Theorem 1.1]{tadic2014irreducibility} for the nonarchimedean places. Since $w_\sigma w_1 = w_1' w_\sigma'$, by the properties of the normalized intertwining operators
\[
    N(w_1,\underline{\Lambda}) = N(w_\sigma^{-1},w_1'w_\sigma'\underline{\Lambda}) N(w_1',w_\sigma'\underline{\Lambda}) N(w_\sigma',\underline{\Lambda}),
\]
and since $N(w_1,\underline{\Lambda})\neq 0$, we have $N(w_\sigma w_1,\underline{\Lambda})\neq 0$. The image of the intertwining operator $N(w_1,\underline{\Lambda})$ is isomorphic to the image $I(\tau,\underline{s})$ of the residue operator \[\mathrm{Res}_{(-\underline{\lambda_m}+s_1,-\underline{\lambda_n}+s_2)}=\prod_{1\leq i\leq m-1}(s_i-s_{i+1}-1)\prod_{1\leq j\leq n-1}(t_j-t_{j+1}-1)E(\cdot,\underline{\Lambda})\mid_{\underline{\Lambda}\to (-\underline{\lambda_m}+s_1,-\underline{\lambda_n}+s_2)}:I(\underline{\Lambda})\rightarrow I(w_1\underline{\Lambda})\] as a submodule in $I(w_1\underline{\Lambda})$. By calculating the poles of the normalized intertwining operator $N(w_\sigma,w_1\underline{\Lambda})$, the residue operator $\mathrm{Res}_{s_1-s_2 = \frac{m+n}{2}-\sigma} E(\cdot,\underline{s})$ kills all the $M(w,\underline{\Lambda})f$ terms in the constant term formula of $E(\cdot,\underline{\Lambda})$ except for those corresponding to the Weyl group element $w=w_2 w_1$ with $w_2$ satisfying
\[
    w_2(m+1)<w_2(m+2)<\ldots<w_2(m+n-\sigma)<w_2(1)<w_2(2)<\ldots<w_2(m).
\]
as well as (\ref*{defprop}). For any such $w_2$, $N(w_2w_\sigma^{-1},w_\sigma w_1\underline{\Lambda})$ maps $I_\sigma$ either isomorphically onto its image or to 0, and thus, the constant term of $\mathrm{Res}_{s_1-s_2 = \frac{m+n}{2}-\sigma}\mathrm{Res}_{(-\underline{\lambda_m}+s_1,-\underline{\lambda_n}+s_2)}E(f,\underline{\Lambda})$ is equal to the constant term of the Eisenstein series constructed from $M(w_\sigma w_1,w_\sigma w_1\underline{\Lambda})f$ as a vector in $I_\sigma$ up to a nonzero constant scalar. Thus, the automorphic representation generated by $\mathrm{Res}_{s_1-s_2 = \frac{m+n}{2}-\sigma}\mathrm{Res}_{(-\underline{\lambda_m}+s_1,-\underline{\lambda_n}+s_2)}E(\cdot,\underline{\Lambda})$ is isomorphic to $I_\sigma$.
\end{proof}

\appendix

\section*{Examples}
\subsection{Case \texorpdfstring{$m=n=2$}{m=n=2}}
For $\sigma = 0$, the two segments $\Delta_1=\left[\frac{1}{2},\frac{3}{2}\right]$ and $\Delta_2=\left[-\frac{3}{2},-\frac{1}{2}\right]$ are \emph{juxtaposed}. The orbits and the corresponding $R(w,\underline{s})$ are listed in the following table:\\
\begin{center}
    \begin{tabular}{ccc}
        Orbit & $\rho_{m,n}$ & $R(w,\frac{m+n}{2}-\sigma+t)$\\
        ${e}$ & $\left\{\frac{1}{2},\frac{3}{2},-\frac{3}{2},-\frac{1}{2}\right\}$ & $1$\\
        ${(1324)}$ & $\left\{\frac{1}{2},-\frac{3}{2},\frac{3}{2},-\frac{1}{2}\right\}$ & $\frac{c^{-t-\frac{5}{2}} L(t+3)}{L(t+4)}$\\
        ${(1342)}$ & $\left\{\frac{1}{2},-\frac{3}{2},-\frac{1}{2},\frac{3}{2}\right\}$ & $\frac{c^{-2 t-4} L(t+2)}{L(t+4)}$\\
        ${(3124)}$ & $\left\{-\frac{3}{2},\frac{1}{2},\frac{3}{2},-\frac{1}{2}\right\}$ & $\frac{c^{-2 t-4} L(t+2)}{L(t+4)}$\\
        ${(3142)}$ & $\left\{-\frac{3}{2},\frac{1}{2},-\frac{1}{2},\frac{3}{2}\right\}$ & $\frac{c^{-3t-\frac{11}{2}} L(t+2)^2}{L(t+3) L(t+4)}$\\
        ${(3412)}$ & $\left\{-\frac{3}{2},-\frac{1}{2},\frac{1}{2},\frac{3}{2}\right\}$ & $\frac{c^{-4 t-6} L(t+1) L(t+2)}{L(t+3) L(t+4)}$
    \end{tabular}
\end{center}

\noindent For $\sigma = 1$, the poles of $r(w,\underline{s})$ in one of the orbits $\{(3142),(3412)\}$ cancel:\\
\begin{center}
    \begin{tabular}{ccc}
        Orbit & $\rho_{m,n}$ & $R(w,\frac{m+n}{2}-\sigma+t)$\\
        ${e}$ & $\{0,1,-1,0\}$ & $1$\\
        ${(1324)}$ & $\{0,-1,1,0\}$ & $\frac{c^{-t-\frac{3}{2}} L(t+2)}{L(t+3)}$\\
        ${(1342)}$ & $\{0,-1,0,1\}$ & $\frac{c^{-2 t-2} L(t+1)}{L(t+3)}$\\
        ${(3124)}$ & $\{-1,0,1,0\}$ & $\frac{c^{-2 t-2} L(t+1)}{L(t+3)}$\\
        ${(3142),(3412)}$ & $\{-1,0,0,1\}$ & $\frac{c^{-3 t-\frac{5}{2}} L(t+1) \left(L(1-t)+L(1+t)\right)}{L(t+2) L(t+3)}$\\
    \end{tabular}
\end{center}

\noindent For $\sigma = 2$, more cancellations occur:\\
\begin{center}
    \begin{tabular}{ccc}
        Orbit & $\rho_{m,n}$ & $R(w,\frac{m+n}{2}-\sigma+t)$\\
        ${e,(3412)}$ & $\left\{-\frac{1}{2},\frac{1}{2},-\frac{1}{2},\frac{1}{2}\right\}$ & $\frac{c^{-2 t} \left(c^{2 t} L(t+1) L(t+2)+L(1-t) L(2-t)\right)}{L(t+1) L(t+2)}$\\
        ${(1324),(1342),(3124),(3142)}$ & $\left\{-\frac{1}{2},-\frac{1}{2},\frac{1}{2},\frac{1}{2}\right\}$ & $\frac{c^{-t-\frac{1}{2}} (L(1-t)+L(t+1))^2}{L(t+1) L(t+2)}$
    \end{tabular}
\end{center}

\subsection{Case \texorpdfstring{$m=2, n=3$}{m=2,n=3}}
For $\sigma = 2$, the $R(w,\underline{s})$'s are displayed in the following table: 
\begin{center}
    \begin{tabular}{ccc}
        Orbit & $\rho_{m,n}$ & $R(w,\frac{m+n}{2}-\sigma+t)$\\
        ${e}$ & $\left\{-\frac{1}{4},\frac{3}{4},-\frac{5}{4},-\frac{1}{4},\frac{3}{4}\right\}$ & $1$\\
        ${(13245)}$ & $\left\{-\frac{1}{4},-\frac{5}{4},\frac{3}{4},-\frac{1}{4},\frac{3}{4}\right\}$ & $\frac{c^{-t-\frac{3}{2}} L(t+2)}{L(t+3)}$\\
        ${(13425),(13452)}$ & $\left\{-\frac{1}{4},-\frac{5}{4},-\frac{1}{4},\frac{3}{4},\frac{3}{4}\right\}$ & $\frac{c^{-2 t-2} (L(1-t)+L(t+1))}{L(t+3)}$\\
        ${(31245),(34512)}$ & $\left\{-\frac{5}{4},-\frac{1}{4},\frac{3}{4},-\frac{1}{4},\frac{3}{4}\right\}$ & $\frac{c^{-4 t-2} \left(c^{2 t} L(t+1) L(t+2)+L(1-t) L(2-t)\right)}{L(t+2) L(t+3)}$\\
        ${(31425),(31452),(34125),(34152)}$ & $\left\{-\frac{5}{4},-\frac{1}{4},-\frac{1}{4},\frac{3}{4},\frac{3}{4}\right\}$ & $\frac{c^{-3 t-\frac{5}{2}} (L(1-t)+L(t+1))^2}{L(t+2) L(t+3)}$.
    \end{tabular}
\end{center}

\subsection{Case \texorpdfstring{$m=3, n=2$}{m=3,n=2}}
Similar to the previous subsection, for $\sigma = 1$, 
\begin{center}
    \begin{tabular}{ccc}
        Orbit & $\rho_{m,n}$ & $R(w,\frac{m+n}{2}-\sigma+t)$\\
        e & $\left\{-\frac{1}{4},\frac{3}{4},\frac{7}{4},-\frac{5}{4},-\frac{1}{4}\right\}$ & $1$\\
        $(12435)$ & $\left\{-\frac{1}{4},\frac{3}{4},-\frac{5}{4},\frac{7}{4},-\frac{1}{4}\right\}$ & $\frac{c^{-t-\frac{5}{2}} L(t+3)}{L(t+4)}$\\
        $(12453)$ & $\left\{-\frac{1}{4},\frac{3}{4},-\frac{5}{4},-\frac{1}{4},\frac{7}{4}\right\}$ & $\frac{c^{-2 t-4} L(t+2)}{L(t+4)}$\\
        $(14235)$ & $\left\{-\frac{1}{4},-\frac{5}{4},\frac{3}{4},\frac{7}{4},-\frac{1}{4}\right\}$ & $\frac{c^{-2 t-4} L(t+2)}{L(t+4)}$\\
        $(14253)$ & $\left\{-\frac{1}{4},-\frac{5}{4},\frac{3}{4},-\frac{1}{4},\frac{7}{4}\right\}$ & $\frac{c^{-3 t-\frac{11}{2}} L(t+2)^2}{L(t+3) L(t+4)}$\\
        $(14523)$ & $\left\{-\frac{1}{4},-\frac{5}{4},-\frac{1}{4},\frac{3}{4},\frac{7}{4}\right\}$ & $\frac{c^{-4 t-6} L(t+1) L(t+2)}{L(t+3) L(t+4)}$\\
        $(41235)$ & $\left\{-\frac{5}{4},-\frac{1}{4},\frac{3}{4},\frac{7}{4},-\frac{1}{4}\right\}$ & $\frac{c^{-3 t-\frac{9}{2}} L(t+1)}{L(t+4)}$\\
        $(41253)$ & $\left\{-\frac{5}{4},-\frac{1}{4},\frac{3}{4},-\frac{1}{4},\frac{7}{4}\right\}$ & $\frac{c^{-4 t-6} L(t+1) L(t+2)}{L(t+3) L(t+4)}$\\
        $(41523),(45123)$ & $\left\{-\frac{5}{4},-\frac{1}{4},-\frac{1}{4},\frac{3}{4},\frac{7}{4}\right\}$ & $\frac{c^{-5 t-\frac{13}{2}} L(t+1) (L(1-t)+L(t+1))}{L(t+3) L(t+4)}$
    \end{tabular}
\end{center}

\subsection{Case $m=3, n=3$}
For $\sigma = 3$, we have:
\begin{center}
    \begin{tabular}{ccc}
        Orbit & $\rho_{m,n}$ & $R(w,\frac{m+n}{2}-\sigma+t)$\\
        ${e,(456123)}$ & $\{-1,0,1,-1,0,1\}$ & $\frac{c^{-6 t} \left(c^{6 t} L(t+1) L(t+2) L(t+3)+L(1-t) L(2-t) L(3-t)\right)}{L(t+1) L(t+2) L(t+3)}$\\
        ${(124356),(451623)}$ & $\{-1,0,-1,1,0,1\}$ & $\frac{c^{-5 t-\frac{3}{2}} \left(c^{4 t} L(t+1) L(t+2)^2+L(1-t) L(2-t)^2\right)}{L(t+1) L(t+2) L(t+3)}$ \\
        $\left\{\begin{array}{c}(124536),(124563),\\(451236),(451263)\end{array}\right\}$ &$\{-1,0,-1,0,1,1\}$& $\frac{c^{-4 t-2} (L(1-t)+L(t+1)) \left(c^{2 t} L(t+1) L(t+2)+L(1-t) L(2-t)\right)}{L(t+1) L(t+2) L(t+3)}$\\
        $\left\{\begin{array}{c}(142356),(145623),\\(412356),(415623)\end{array}\right\}$ &$\{-1,-1,0,1,0,1\}$&$\frac{c^{-4 t-2} (L(1-t)+L(t+1)) \left(c^{2 t} L(t+1) L(t+2)+L(1-t) L(2-t)\right)}{L(t+1) L(t+2) L(t+3)}$\\
        $\left\{\begin{array}{c}
            (142536),(142563),\\(145236),(145263)\\
            (412536),(412563),\\(415236),(415263)
        \end{array}\right\}$ &$\{-1,-1,0,0,1,1\}$&$\frac{c^{-3 t-\frac{5}{2}} (L(1-t)+L(t+1))^3}{L(t+1) L(t+2) L(t+3)}$
    \end{tabular}
\end{center}

\bibliography{main}
\bibliographystyle{alpha}

\end{document}